\documentclass[12pt]{amsart}
\pdfoutput=1
\usepackage{bm}
\usepackage{graphicx}
\usepackage{mathpazo}
\usepackage{latexsym}              
\usepackage{amssymb}               
\usepackage{amsmath}               
\usepackage{amsthm}                
\usepackage{booktabs}              
\usepackage[margin=1in]{geometry}
\usepackage[usenames,dvipsnames]{xcolor} 
\usepackage[breaklinks=true,
            colorlinks, 
            linkcolor=MidnightBlue, 
            citecolor=MidnightBlue,
            urlcolor =BlueViolet,
           ]{hyperref}

\usepackage[all]{xy}
\usepackage{colonequals}
\usepackage{comment}

\newtheorem{theorem}{Theorem}[section]
\newtheorem{corollary}[theorem]{Corollary}
\newtheorem{definition}[theorem]{Definition}
\newtheorem{lemma}[theorem]{Lemma}
\newtheorem{proposition}[theorem]{Proposition}

\theoremstyle{definition}
\newtheorem{claim}[theorem]{Claim}
\newtheorem{remark}[theorem]{Remark}

\DeclareMathOperator{\Aut}{Aut}
\DeclareMathOperator{\Disc}{Disc}
\DeclareMathOperator{\End}{End}
\DeclareMathOperator{\GL}{GL}
\DeclareMathOperator{\Hom}{Hom}
\DeclareMathOperator{\Id}{Id}
\DeclareMathOperator{\image}{Image}
\DeclareMathOperator{\Res}{Res}
\DeclareMathOperator{\SL}{SL}
\DeclareMathOperator{\Spec}{Spec}

\newcommand{\BF}{{\mathbb{F}}}
\newcommand{\PP}{{\mathbb{P}}}
\newcommand{\BQ}{{\mathbb{Q}}}
\newcommand{\BZ}{{\mathbb{Z}}}
\newcommand{\alphatwo}{\boldsymbol{\alpha}_2}

\newcommand{\psibar}{\overline{\psi}}

\newcommand{\ord}{{\mathcal{O}}}
\newcommand{\calW}{{\mathcal{W}}}

\newcommand{\frakp}{{\mathfrak{p}}}

\newcommand{\Arr}{A_\textup{rr}}
\newcommand{\Arl}{A_\textup{rl}}
\newcommand{\Alr}{A_\textup{lr}}
\newcommand{\All}{A_\textup{ll}}
\newcommand{\blankrel}{\mathrel{\phantom{=}}}
\newcommand{\pz}{\phantom{0}}
\newcommand{\eps}{\varepsilon}

\renewcommand{\bar}\mybar
\renewcommand{\hat}\widehat
\renewcommand{\tilde}\widetilde

\newcommand{\mybar}[1]{
  \mathchoice
  {#1\llap{$\overline{\phantom{\displaystyle\rm#1}}$}}
  {#1\llap{$\overline{\phantom{\textstyle\rm#1}}$}}
  {#1\llap{$\overline{\phantom{\scriptstyle\rm#1}}$}}
  {#1\llap{$\overline{\phantom{\scriptscriptstyle\rm#1}}$}}
}

\newcommand{\col}{\,{:}\,}

\makeatletter

\def\@marginparreset{\marginparstyle}
\def\marginparstyle{\SMALL\normalfont\raggedright\openup-2pt }
\long\def \@savemarbox #1#2{%
 \global\setbox #1%
     \vtop{%
       \hsize\marginparwidth
       \@parboxrestore
       \reset@font
       \@setnobreak
       \@setminipage
       \@marginparreset
       #2%
       \par
       \global\@minipagefalse
       }%
}

\DeclareRobustCommand\marginparhere[2][0pt]{%
 \ifhmode\unskip\fi
 \ifmmode\ssty\mathclose{\fi
   \rlap{\hskip\marginparsep\smash{%
     \vtop to 0pt{\marginparstyle \hsize\marginparwidth 
       \leftskip=#1 \rightskip=-#1 plus20pt \noindent#2\vss}}}%
 \ifmmode}\fi
}

\DeclareRobustCommand{\redden}{\@ifnextchar*
 {\@latex@error{{redden*} is only an environment}}
 {\@ifnextchar[{\r@dden}{\r@dd@n}}}
\def\r@dden[#1]{\ifmmode\@mperr{math}\else\ifinner\@mperr{inner}%
 \else\leavevmode\marginpar{\leavevmode#1\endgraf}\fi\fi\r@dd@n}
\def\r@dd@n{\def\reserved@a{redden}
 \ifx\@currenvir\reserved@a\redd@n\bgroup\ignorespaces
 \else\expandafter\redd@n\fi}
\def\endredden{\unskip\egroup}
\def\@mperr#1{\@latex@warning{redden in #1 mode: no marginpar possible}}
\def\redd@n#1{\leavevmode{\color{red}#1}}

\makeatother

\begin{document}

\title{Purely inseparable Richelot isogenies}

\author{Bradley W.~Brock}
\address{Center for Communications Research,
         805 Bunn Drive, 
         Princeton, NJ 08540-1966 USA}
\email{\href{mailto:"Bradley W. Brock" <brock@idaccr.org>}{brock@idaccr.org}}

\author{Everett W.~Howe}
\address{Unaffiliated mathematician, San Diego, CA, USA}
\urladdr{\url{http://ewhowe.com}}
\email{\href{mailto:"Everett W. Howe" <however@alumni.caltech.edu>}{however@alumni.caltech.edu}}

\date{25 March 2025}



\begin{abstract}
We show that if $C$ is a supersingular genus-$2$ curve over an algebraically 
closed field of characteristic~$2$, then there are infinitely many Richelot
isogenies starting from (the Jacobian of)~$C$.  This is in contrast to 
what happens with nonsupersingular curves in characteristic~$2$, or to arbitrary
curves in characteristic not~$2$: In these situations, there are at most fifteen
Richelot isogenies starting from a given genus-$2$ curve.

More specifically, we show that if $C_1$ and $C_2$ are two arbitrary 
supersingular genus-$2$ curves over an algebraically closed field of 
characteristic~$2$, then there are exactly sixty Richelot isogenies from $C_1$ 
to~$C_2$, unless either $C_1$ or $C_2$ is isomorphic to the curve 
$y^2 + y = x^5$.  In that case, there are either twelve or four Richelot 
isogenies from $C_1$ to~$C_2$, depending on whether $C_1$ is isomorphic 
to~$C_2$. (Here we count Richelot isogenies up to isomorphism.) We give explicit
constructions that produce all of the Richelot isogenies between two 
supersingular curves.
\end{abstract}

\maketitle

\section{Introduction}
Let $k$ be an algebraically closed field of characteristic~$2$. In this paper we
prove that there are Richelot isogenies connecting every two supersingular 
genus-$2$ curves over~$k$. More specifically:

\begin{theorem}
\label{T:main}
Let $C_1$ and $C_2$ be two supersingular genus-$2$ curves over~$k$ with
polarized Jacobians $(J_1,\lambda_1)$ and $(J_2,\lambda_2)$. If neither $C_1$
nor $C_2$ is isomorphic to the curve $y^2 + y = x^5$, then there are exactly
sixty Richelot isogenies from $(J_1,\lambda_1)$ to $(J_2,\lambda_2)$, up to 
isomorphism.  If exactly one of $C_1$ or $C_2$ is isomorphic to the special
curve, then there are twelve Richelot isogenies from $(J_1,\lambda_1)$ to
$(J_2,\lambda_2)$, up to isomorphism.  If both $C_1$ and $C_2$ are isomorphic
to the special curve, then there are four Richelot isogenies from 
$(J_1,\lambda_1)$ to $(J_2,\lambda_2)$, up to isomorphism.
\end{theorem}

In Section~\ref{S:explicit} we provide explicit constructions that give all of
the Richelot isogenies connecting the Jacobians of two supersingular curves.

To put our result in context, let us give some background information about
Richelot isogenies. Let $k$ be an algebraically closed field of arbitrary 
characteristic and let $J_1$ and $J_2$ be abelian surfaces over~$k$ with
principal polarizations $\lambda_1$ and~$\lambda_2$. (We view a polarization of
an abelian variety $A$ as an isogeny from $A$ to its dual variety, rather than
as a line bundle.) A \emph{Richelot isogeny} from $(J_1,\lambda_1)$ to
$(J_2,\lambda_2)$ is an isogeny $\varphi\colon J_1 \to J_2$ that fits into a
diagram 
\begin{equation}
\label{EQ:Diagram1}
\begin{gathered}
\xymatrix{
J_1\ar[rr]^{2\lambda_1}\ar[d]_{\varphi} && \hat{J_1}                       \\
J_2\ar[rr]^{\lambda_2}                  && \hat{J_2}\ar[u]_{\hat{\varphi}}
}
\end{gathered}
\end{equation}
where $\hat{\varphi}$ is the dual isogeny of~$\varphi$. If $C_1$ and $C_2$ are
genus-$2$ curves over~$k$, a \emph{Richelot isogeny} from $C_1$ to $C_2$ is
defined to be a Richelot isogeny from the Jacobian of $C_1$ to the Jacobian 
of~$C_2$.

Up to automorphisms of $(J_2,\lambda_2)$, a Richelot isogeny is determined by
its kernel in~$J_1$.  In characteristic not~$2$, this kernel is an order-$4$ 
subgroup of the $2$-torsion of $J_1(k)$ that is isotropic with respect to the
Weil pairing on~$J_1[2]$. Conversely, every maximal isotropic subgroup of 
$J_1[2](k)$ gives rise to an isogeny from $(J_1,\lambda_1)$ to some 
principally-polarized variety $(J_2,\lambda_2)$ 
(see~\cite[Proposition~16.8, p.~135]{Milne1986}).

We say that two Richelot isogenies $\varphi$ and $\psi$ from  $(J_1,\lambda_1)$
to $(J_2,\lambda_2)$ are \emph{isomorphic} to one another if there are
automorphisms $\beta_1$ of $(J_1,\lambda_1)$ and $\beta_2$ of $(J_2,\lambda_2)$
such that $\psi = \beta_2 \circ \varphi\circ \beta_1$. Thus we see that in
characteristic not $2$ there is a bijection between the set of isomorphism
classes of Richelot isogenies starting from $(J_1,\lambda_1)$ and the set of 
orbits of maximal isotropic subgroups of $J_1[2](k)$ under the action of the
automorphism group of $(J_1,\lambda_1)$.  Typically this automorphism group acts
trivially on the set of maximal isotropic subgroups, and in this case the number
of Richelot isogenies starting from $(J_1,\lambda_1)$ is just the number of
maximal isotropic subgroups of~$J_1[2](k)$.

An easy calculation shows that every $4$-dimensional $\BF_2$-vector space with a
nondegenerate alternating pairing has exactly $15$ maximal isotropic subgroups.
Thus we see that for a typical principally-polarized surface $(J_1,\lambda_1)$
in characteristic not~$2$, there are $15$ nonisomorphic Richelot isogenies
starting from~$(J_1,\lambda_1)$. 

In characteristic $2$ life is a little different, because the $2$-torsion of 
$J_1$ can no longer be understood simply in terms of the points of order $2$ 
in~$J_1(k)$.  Instead, we must consider $J_1[2]$ as a group scheme. Indeed, the
bijection we mentioned two paragraphs ago is really between the set of 
isomorphism classes of Richelot isogenies starting from $(J_1,\lambda_1)$ and
the set of orbits of maximal isotropic subgroup schemes of $J_1[2]$ under the 
action of the automorphism group $\Aut(J_1,\lambda_1)$.  Let us see what this
means in characteristic~$2$.

Suppose that $(J_1,\lambda_1)$ is a principally-polarized abelian surface over
an algebraically closed field $k$ of characteristic~$2$. The rank\footnote{
     The \emph{rank} of a finite group scheme $G$ over a field $k$ is the
     dimension of the $k$-algebra $R$ such that $G\cong \Spec R$ as a scheme;
     the rank of a finite group scheme generalizes the idea of the 
     \emph{order} of a finite group.} 
of the group scheme $J_1[2]$ is~$16$, and $J_1[2]$ (like all finite commutative
group schemes over perfect fields) can be decomposed into a product
\[
J_1[2] = \Arr \times \Arl \times \Alr \times \All,
\]
where $\Arr$ is a reduced group scheme with reduced dual, $\Arl$ is a reduced 
group scheme with local dual, $\Alr$ is a local group scheme with reduced dual,
and $\All$ is a local group scheme with local dual (\cite[p.~17]{Manin1963},
\cite[Corollary, p.~52]{Waterhouse1979}, \cite[p.~136]{Mumford1970}). Since
$J_1[2]$ is a $2$-torsion group scheme in characteristic~$2$, its
reduced-reduced factor is trivial. And since $J_1[2]$ is self-dual (via the Weil
pairing, obtained from the principal polarization~$\lambda_1$), its
reduced-local and local-reduced parts are dual to one another. Thus, there are
three possibilities: 
\begin{itemize}
\item $\All$ has rank $1$ and $\Arl$ and $\Alr$ have rank~$4$
      (the \emph{ordinary} case);
\item $\All$ has rank $4$ and $\Arl$ and $\Alr$ have rank~$2$
      (the \emph{almost ordinary} case);
\item $\All$ has rank $16$ and $\Arl$ and $\Alr$ have rank~$1$
      (the \emph{supersingular} case).
\end{itemize}

Consider the ordinary case, where $J_1[2]$ is isomorphic to the product of a 
rank-$4$ reduced group scheme $\Arl$ with a rank-$4$ local group scheme~$\Alr$.
The kernel of a Richelot isogeny starting at $J_1$ can be also be written as a
product $R\times L$ of a reduced group scheme $R$ with a local group scheme~$L$,
and there are three possible shapes for these groups:
\begin{itemize}
\item $R = 0$ and $L = \Alr$;
\item $R = \Arl$ and $L = 0$;
\item $R(k)$ is $\{0,T\}$ for one of the three points $T$ of order
      $2$ in~$J_1(k)$, and $L$ is the unique rank-$2$ subgroup scheme
      of $\Alr$ that pairs trivially with~$R$ under the Weil pairing.
\end{itemize}
The first possibility is the Frobenius isogeny starting at $(J_1,\lambda_1)$, 
and the second is the Verschiebung. The third possibility gives rise to three 
Richelot isogenies, corresponding to the three points of order~$2$ in~$J_1(k)$.
These isogenies can be understood by lifting to characteristic $0$ via the 
Serre--Tate canonical lift. (We note also that lifting to characteristic $0$
suggests a natural way of assigning multiplicities to the three types of
ordinary Richelot isogenies in characteristic $2$ so that the weighted number of
such isogenies is~$15$.)

For the almost-ordinary case, we note that a result of 
Manin~\cite[Theorem~4.1, pp.~72--73]{Manin1963} shows that the rank-$4$ group 
scheme $\All$ that appears in the product decomposition of $J_1[2]$ is 
isomorphic to the $2$-torsion subgroup $E[2]$ of the unique supersingular 
elliptic curve~$E$ over~$k$, and so there is a unique isotropic subgroup $S$ of
$\All$ of rank~$2$. Thus, the only maximal isotropic subgroups of 
${\Arl\times\Alr\times\All}$ are $\Arl\times S$ and $\Alr\times S$, so there are
only two Richelot isogenies starting from $(J_1,\lambda_1)$. The one with kernel
$\Alr\times S$ is the Frobenius, and the one with kernel $\Arl\times S$ is 
Verschiebung. 

Note that in every case we have discussed so far, a given polarized abelian
surface is the source for finitely many Richelot isogenies because there are
only finitely many maximal isotropic subgroup schemes of the $2$-torsion:
fifteen in characteristic not~$2$, five for ordinary surfaces in
characteristic~$2$, and two for almost-ordinary surfaces in characteristic~$2$.
That is what makes Theorem~\ref{T:main} somewhat surprising: Since there are
Richelot isogenies between every pair of supersingular curves, there are
\emph{infinitely many} Richelot isogenies starting from a given supersingular
curve.

Our definition of a Richelot isogeny is of course a modern one. Richelot's 
original papers~\cite{Richelot1836, Richelot1837}, published in 1836 
and~1837, were concerned with the evaluation of ``ultra-elliptic integrals''.
In~1865, K\"onigsberger~\cite{Konigsberger1865} interpreted Richelot's work in
terms of duplication formul\ae\ for two-variable theta functions. In 
characteristic~$0$, this formulation is essentially the same as our formulation
in terms of isogenies of abelian varieties. See the article by Bost and
Mestre~\cite{BostMestre1988} or the book by Cassels and
Flynn~\cite[Ch.~9]{CasselsFlynn1996} for more information.

This paper is structured as follows.  In Section~\ref{S:review} we review a few
facts about supersingular genus-$2$ curves in characteristic~$2$.  In 
Section~\ref{S:theorem} we restate our main theorem in slightly different terms,
together with three lemmas used in its proof. In Section~\ref{S:proof} we prove
the theorem, and in Section~\ref{S:explicit} we give explicit constructions 
that, in every case, provide all of the Richelot isogenies between two curves.

Some of our arguments depend on computations that can be accomplished with a 
computer algebra system. We have included a Magma file with the auxiliary 
material for the arXiv version of this paper; this file includes code to verify
some of our statements. We indicate in the text which statements have such
verifications.

\section{Supersingular genus-\texorpdfstring{$2$}{2} curves in characteristic~\texorpdfstring{$2$}{2}}
\label{S:review}
In this section we present some background material and easily-proven results
about supersingular genus-$2$ curves in characteristic~$2$.

Recall that the \emph{Igusa invariants} 
$[J_2\col J_4\col J_6\col J_8\col J_{10}]$ of a genus-$2$ curve over
a field~$k$ form an element of a weighted projective space over~$k$, where each
$J_{i}$ has weight~$i$, and where we have $J_2 J_6 = J_4^2 + 4J_8$ and
$J_{10}\neq 0$.

\begin{lemma}
\label{L:models}
Let $C$ be an arbitrary genus-$2$ curve over an algebraically closed field $k$
of characteristic~$2$, with Igusa invariants 
$[J_2\col J_4\col J_6\col J_8\col J_{10}]$. 
\begin{itemize}
\item[\textup{1.}] $C$ is supersingular if and only if 
      $J_2 = J_4 = J_6 = 0$.
\item[\textup{2.}] If $C$ is supersingular and $J_8 \neq 0$, then
     \begin{itemize}
     \item[\textup{(a)}] $C$ is isomorphic to the curve 
          $y^2 + y = Ax^5 + Ax^3$, where $A^{16} = J_8^5/J_{10}^4$.
     \item[\textup{(b)}] $\Aut C$ is a group of order~$32$. 
          In particular, it is the nonsplit extension of $(\BZ/2\BZ)^4$ by
          $\BZ/2\BZ$ that is identified as \textup{SmallGroup(32,50)}
          in the Magma/GAP database of small groups, and by name as
          \textup{\texttt{C4.C2{\textasciicircum}3}}.
     \item[\textup{(c)}] More specifically\textup{:} For every root $t$ of
          \[
           \qquad\qquad T^{16} + T^8 + A^{-2} T^2 + A^{-3} T 
           =  T ( T^5 + T + A^{-1} ) ( T^{10} + T^6 + A^{-1} T^5 + A^{-2}),
          \]
          there are two polynomials $f$ of degree at most $2$ such that
          $(x,y)\mapsto(x+t,y+f(x))$ is an automorphism of~$C$. Every
          automorphism of $C$ arises in this way. The automorphisms with $t=0$
          are the identity and the hyperelliptic involution~$\iota$\textup{;}
          the automorphisms with $t$ a root of $T^5 + T + A^{-1}$ have
          order~$2$\textup{;} the automorphisms with $t$ a root of 
          $T^{10} + T^6 + A^{-1} T^5 + A^{-2}$ have order~$4$, and their squares
          are the hyperelliptic involution. Two automorphism of order~$4$ are
          conjugate if and only if they come from the same value of~$t$.
     \end{itemize}
\item[\textup{3.}] If $C$ is supersingular and $J_8 = 0$, then
     \begin{itemize}
     \item[\textup{(a)}] $C$ is isomorphic to the curve $y^2 + y = x^5$.
     \item[\textup{(b)}] $\Aut C$ is a group of order $160$. In particular, it
          is the semidirect product of $\BZ/5\BZ$ acting nontrivially on the
          group from statement \textup{2(b)}. 
          This group is identified as \textup{SmallGroup(160,199)} in the
          Magma/GAP database of small groups, and by name as
          \textup{\texttt{(C4.C2{\textasciicircum}3):C5}}.
     \item[\textup{(c)}] More specifically\textup{:} for every fifth root of 
          unity $\zeta$ and every root $t$ of 
          \[ T^{16} + T = T(T^5+1)(T^{10}+T^5+1), \]
          there are two polynomials $f$ of degree at most $2$ such that 
          $(x,y)\mapsto(\zeta x+t,y+f(x))$
          is an automorphism of~$C$. Every automorphism of $C$ arises in this
          way. The automorphisms with $\zeta\neq 1$ have order divisible 
          by~$5$\textup{;} the automorphisms with $\zeta=1$ and $t=0$ are the
          identity and the hyperelliptic involution~$\iota$\textup{;} the 
          automorphisms with $\zeta=1$ and $t^5 = 1$ have order~$2$\textup{;} 
          the automorphisms with $\zeta=1$ and with $t$ a root of 
          $T^{10} + T^5 + 1$ have order~$4$, and their squares are the
          hyperelliptic involution. Two automorphism of order~$4$ are conjugate
          if and only if they come from values of $t$ having the same fifth power.
     \end{itemize}
\end{itemize}
\end{lemma}

\begin{proof}
First we show that every supersingular genus-$2$ curve over $k$ can be written
in the form $y^2 + y = f$ for a quintic polynomial~$f$.

Every genus-$2$ curve over $k$ has a model, nonsingular in the affine plane, of
the form ${y^2 + hy = f}$ for polynomials $h,f\in k[x]$ with $f$ of degree $5$
and $h$ nonzero of degree at most~$2$. (The nonsingularity is equivalent to $h$
and $(h')^2f + (f')^2$ being coprime.)  The desingularization of this model has
a single point at infinity, which we denote~$\infty$. If $h$ is nonconstant, say
with a root~$a$, then the degree-$0$ divisor $D = (a,\sqrt{f(a)}) - \infty$
represents a $2$-torsion point on the Jacobian, because $2D$ is the divisor of 
the function~$x-a$. Therefore, for a supersingular curve~$C$, the polynomial $h$
must be a nonzero constant. By scaling $y$ appropriately, we may assume 
that~${h = 1}$. 

By explicit computation\footnote{
     See Verification 2.1a in the Magma file.
} (which we leave to the reader and their computer algebra system) we see that
the Igusa invariants of a curve of the form ${y^2 + y = \text{(quintic)}}$ 
satisfy $J_2 = J_4 = J_6 = 0$. These conditions define an irreducible 
$1$-dimensional subvariety of the moduli space of curves, and we have just shown
that it contains the subvariety of supersingular curves. But the moduli space of
supersingular principally-polarized abelian surfaces is one-dimensional and 
closed in the moduli space of all principally-polarized abelian 
surfaces~\cite[Theorem~7(i), p.~163]{Koblitz1975}, so every curve with 
$J_2 = J_4 = J_6 = 0$ must be supersingular.

This completes the proof of the first statement. However, the appeal to a deep
theorem about the dimension of components of moduli spaces of abelian varieties
with a given $p$-rank might be unsatisfying, especially since there are 
elementary arguments that will serve instead. So let us show more explicitly
that every genus-$2$ curve with $J_2 = J_4 = J_6 = 0$ must be supersingular.
The argument also foreshadows reasoning about genus-$4$ curves that appears
in Section~\ref{SS:parameters}.

Let $P = (r,s)$ be any affine point on the hyperelliptic curve $y^2 + y = x^5$
with $r\neq0$ and let $P' = (r,s+1)$ be the image of $P$ under the hyperelliptic
involution. Set $b = (r^5 + 1)/r$ and consider the curve
\[
C_b\colon y^2 + y = x^5 + bx^3.
\]
We check\footnote{
     See Verification 2.1b in the Magma file for many statements in the
     following paragraph.
} that there is an involution $\alpha\colon C_b\to C_b$ given by 
\[
(x,y) \mapsto (x + r^2, y + r x^2 + r^3 x + s),
\]
that the functions $u = x^2 + r^2 x$ and $v = y + (x/r^2)(r x^2 + r^3 x + s)$
are stable under this involution, and that these functions satisfy
$v^2 + v = u^3/r^2 + (s/r^4)u,$ so that the quotient of $C_b$ by the involution
$\alpha$ is the supersingular elliptic curve $E_P$ given by 
$y^2 + y = x^3/r^2 + (s/r^4)x$. Repeating this computation with $P'$ (that is,
replacing $s$ with $s+1$), we construct another involution $\alpha'$ of~$C_b$,
and we find the quotient of $C_b$ by this involution is the supersingular 
elliptic curve $E_{P'}$ given by $y^2 + y = x^3/r^2 + ((s+1)/r^4)x$. We compute
that $\alpha$ and $\alpha'$ commute with one another, and we note that their
product is the hyperelliptic involution $\iota$ of~$C_b$. It follows that we
have a diagram
\[
\xymatrix{
&& C_b\ar[rrd]^{\langle \alpha'\rangle}\ar[lld]_{\langle\alpha\rangle}\ar[d]_{\langle\iota\rangle}&&\\
E_P\ar[rrd] && \PP^1\ar[d] && E_{P'}\ar[lld]\\
&& \PP^1 && 
}
\]
that exhibits $C_b$ as a Galois $V_4$ extension of~$\PP^1$. This diagram shows
that the Jacobian of $C_b$ is isogenous to the product $E_P \times E_{P'}$, so
$C_b$ is supersingular. The vector of Igusa invariants of $C_b$ is 
$[0\col 0\col 0\col b^8 \col 1]$. Thus, every genus-$2$ curve over $k$ with 
$J_2 = J_4 = J_6 = 0$ is supersingular.

Next we note that two vectors of Igusa invariants 
$[0 \col 0 \col 0 \col J_8\col J_{10}]$ and 
$[0 \col 0 \col 0 \col J'_8\col J'_{10}]$ (with $J_{10}$ and $J'_{10}$ nonzero)
are equal if and only if $J_8^5/J_{10}^4 = (J'_8)^5 / (J'_{10})^4$, so 
supersingular curves over $k$ are classified up to isomorphism by the invariant
$I\colonequals J_8^5/J_{10}^4$.

Let us turn to statement~(2). Given any nonzero~$A$, we compute\footnote{
     See Verification 2.1c in the Magma file.
} that the Igusa invariants of the curve $C_A\colon y^2 + y = Ax^5 + Ax^3$ are
$[0 \col 0 \col 0 \col A^8 \col A^6]$, so the invariant $I(C_a)$ is equal
to~$A^{16}$, which is equal to the invariant of~$I(C)$. Thus, $C$ is isomorphic
to~$C_A$.

Suppose $\alpha$ is an automorphism of the curve~$C_A$. The hyperelliptic
involution $\iota$ of $C_A$ is central in the automorphism group, so $\alpha$
commutes with~$\iota$, and it is easy to check that therefore $\alpha$ must be
of the form $(x,y)\mapsto (g(x), y + f(x))$ for rational functions $f(x)$ and
$g(x)$ in~$k(x)$. The rational function $g(x)$ must define an automorphism of
$\PP^1$ that fixes the unique ramification point of the double cover 
$C_A \to \PP^1$ (which is $\infty$) so $g(x) = ax + t$ for some $a,t\in k$ with
$a\ne 0$. From the condition that 
\[
(y + f(x))^2 + (y + f(x)) = A(ax+t)^5 + A(ax+t)^3
\]
we find that\footnote{
     See Verification 2.1d in the Magma file.
}  
\[
f(x)^2 + f(x) = A(a^5-1) x^5 + Aa^4t x^4 
 + A(a^3-1)x^3 + Aa^2tx^2 + Aa(t^4 + t^2)x + A(t^5 + t^3).
\]
This shows that $a^5 = a^3 = 1$ so that $a = 1$, and also that the rational
function $f(x)$ is a polynomial of degree at most~$2$. Writing 
$f(x) = f_2 x^2 + f_1 x + f_0$, we find that\footnote{
     See Verification 2.1e in the Magma file.
} 
\[
f_2^2 x^4 + (f_1^2 + f_2) x^2 + f_1 x + (f_0^2 + f_0)
=
A t x^4 + A t x^2 + A (t^4 + t^2) x + A (t^5 + t^3),
\]
so that
\begin{align*}
f_0^2 + f_0 &= At^5 + At^3\\
        f_1 &= At^4 + At^2\\
        f_2 &= f_1^2 + A t = A^2 t^8 + A^2 t^4 + A t\\
      f_2^2 &= A t.
\end{align*}
The first of these equations always has two solutions. The remaining three can 
be solved if and only if $(A^2 t^8 + A^2 t^4 + A t)^2 = A t$, which can be 
rewritten as\footnote{
     See Verification 2.1f in the Magma file for this statement and the ones
     in the following paragraph.
} 
\[
A^4 t^{16} + A^4 t^8 + A^2 t^2 +  A t = 0,
\]
or
\[
0 
= t^{16} + t^8 + A^{-2} t^2 +  A^{-3} t
= t  (t^5 + t + A^{-1})  (t^{10} + t^6 + A^{-1}t^5 + A^{-2}).
\]
Thus, every automorphism gives a root $t$ of the polynomial in statement~2(c), 
and every such root $t$ gives two polynomials $f(x) = f_2 x^2 + f_1 x + f_0$ 
such that $(x,y)\mapsto (x+t, y + f(x))$ is an automorphism. Note that this 
already shows that $\#\Aut C = 32$.

When $t=0$ the automorphisms one obtains are the identity and the hyperelliptic
involution~$\iota$. Suppose 
$\alpha\colon (x,y)\mapsto (x+t, y + f_2 x^2 + f_1 x + f_0)$ is an automorphism
with $t\ne 0$. We compute that $\alpha^2$ is given by
$(x,y) \mapsto (x, y + f_2 t^2 +  f_1 t)$, and using the formul\ae\ for $f_1$
and $f_2$ we see that 
\[
f_2 t^2 +  f_1 t  = A^2 t^{10} + A^2 t^6 + A t^5 = A t^5 (At^5 + At + 1).
\]
If $t^5 + t + A^{-1} = 0$ then $f_2 t^2 +  f_1 t = 0$ and $\alpha^2$ is the 
identity. If $A^2 t^{10} + A^2 t^6 + A t^5 + 1 = 0$ then  $f_2 t^2 +  f_1 t = 1$
and $\alpha^2 = \iota$, so  $\alpha$ has order~$4$.

It is clear that conjugate automorphisms have the same value of~$t$. Let 
$\alpha$ be any an automorphism of order~$4$. Consider the automorphisms $\beta$
of order $4$ that do \emph{not} conjugate $\alpha$ to~$\alpha^3$; for such a
$\beta$ we must have $\alpha\beta = \beta\alpha$, and we see that 
\[(\alpha\beta)^2 = \alpha^2\beta^2 = \iota^2 = 1.\]
This shows that $\beta\mapsto\alpha\beta$ is an injective map from the set of 
$\beta$ of order $4$ that commute with $\alpha$ to the set of automorphisms of
order at most~$2$. Since there are $20$ automorphisms of order $4$ and only $12$
of order at most~$2$, we see that there must be a $\beta$ of order $4$ that does
\emph{not} commute with~$\alpha$, and that therefore conjugates $\alpha$ 
to~$\alpha^3$. Thus, there is a bijection between the conjugacy classes of
automorphisms of order $4$ and the roots of 
$T^{10} + T^6 + A^{-1} T^5 + A^{-2}$. This proves statement~2(c).

We have already seen that $\#\Aut C = 32$. To determine the structure of 
$\Aut C$ it suffices to look at a specific curve of this form over a finite
field, because we gain no automorphisms when specializing from the generic
case. The curve over $\BF_{2^2}$ where $A$ is a primitive cube root of unity has
all of its automorphisms rational over $\BF_{2^{10}}$, so this is a convenient
example to choose. We find that the center of $\Aut C$ is generated by the
hyperelliptic involution. We already know that every element of $\Aut C$ squares
to either the identity or the hyperelliptic involution, so the quotient of 
$\Aut C$ by its center is $(\BZ/2\BZ)^4$. Running through the database of groups
of order~$32$, we find\footnote{
     See Verification 2.1g in the Magma file.
} only two groups with this property, and only one of them has eleven elements
of order~$2$, namely \texttt{C4.C2\textasciicircum3}. This proves
statement~2(b).

The proof of statement~(3) is analogous, and we leave it to the 
reader.\footnote{
     See Verification 2.1h in the Magma file.
} The main difference is that instead of finding that $a^5 = a^3 = 1$, we only
have that $a^5 = 1$, so that $g(x)$ can be of the form $\zeta x + t$ for a fifth
root of unity~$\zeta$. Then one can check that the map from $\Aut C$ to the
fifth roots of unity that takes an automorphism to the coefficient of $x$ in 
$g(x)$ is a homomorphism, and that its kernel is isomorphic to the group from 
statement~2(b).
\end{proof}

The invariant we used in the proof of Lemma~\ref{L:models} is useful enough to
deserve a name.

\begin{definition}
If $C$ is a supersingular genus-$2$ curve in characteristic~$2$, with Igusa
invariants $[0 \col 0 \col 0 \col J_8\col J_{10}]$, we define the
\emph{supersingular invariant} of $C$ to be the quantity
$I(C) = J_8^5 / J_{10}^4$.
\end{definition}

The models of curves that we use in Lemma~\ref{L:models} have the benefit that
they show that if $k$ is a perfect field of characteristic~$2$, then for every
$a\in k$ there is a supersingular genus-$2$ curve over $k$ whose supersingular
invariant is equal to~$a$. A disadvantage of the models is that the curve with
supersingular invariant $0$ has a different form than all of the other curves.
There is a different family of curves that can be more convenient to use over an
algebraically closed field when we would like to treat all supersingular
curves at the same time, without making significantly different arguments for
the curve with invariant~$0$.

\begin{lemma}
Let $k$ be a field of characteristic~$2$. For every element $B\in k$, the curve
$y^2 + y = x^5 + Bx^3$ is supersingular and has supersingular invariant equal 
to~$B^{40}$. \qed
\end{lemma}

\begin{remark}
\label{R:alternateform}
Note that if $B\neq 0$ and if we let $A$ be an element of the algebraic closure
of $k$ with $A^2 = B^5$, then replacing $x$ with $B^{1/2}x$ in the equation for
the curve in the lemma shows that the curve is isomorphic to 
$y^2 + y = Ax^5 + Ax^3.$ Also, the automorphisms of the curve in the lemma
lie over the automorphisms $x\mapsto x + t$ of $\PP^1$, where $t$ is a root
of $T^{16} + B^4 T^8 + B^2 T^2 +  T$.
\end{remark}

We close with some useful statements about the finite simple subgroup schemes
of the square of the unique supersingular elliptic curve $E$ over~$k$. These
statements can be found in~\cite[\S2]{IbukiyamaKatsuraOort1986}.

Let $\alphatwo$ be the unique simple local-local group scheme over~$k$; the
group scheme $\alphatwo$ is the kernel of Frobenius on the additive group 
over~$k$.  Note that $\Hom(\alphatwo,\alphatwo)(k) = k$, and that there is a
unique copy of $\alphatwo$ in~$E$.  Fix an embedding $e\colon\alphatwo\to E$.

Given two elements $i,j\in k$, not both~$0$, we get an embedding
\[
\xymatrix{
          \alphatwo \ar[rr]^{(i,j)} & & 
          \alphatwo\times\alphatwo\ar[rr]^{e\times e} & & 
          E\times E
}
\]
of $\alphatwo$ in $E \times E$. The images we get from a pair $(i,j)$ and a pair
$(i',j')$ are the same if and only if $[i\col j] = [i'\col j']$ as elements
of~$\PP^1(k)$. Furthermore,
the image does not depend on the choice of~$e$.

The quotient of $E \times E$ by $[i\col j](\alphatwo)$ is isomorphic to
$E \times E$ if and only if $[i\col j]$ lies in $\PP^1(\BF_4)$.  Otherwise, the
quotient contains a unique copy of~$\alphatwo$. The quotient, divided by this
unique copy of $\alphatwo$, is~$E \times E$.

The Jacobian of a supersingular genus-$2$ curve over $k$ is isomorphic to 
$E \times E / [i\col j](\alphatwo)$ for some $[i\col j] \in \PP^1(k)$ that does
not lie in $\PP^1(\BF_4)$.

\section{The theorem and three lemmas}
\label{S:theorem}
Let $k$ be an algebraically closed field of characteristic~$2$. Using the
terminology from the preceding section, we can restate our main theorem as
follows:

\begin{theorem}
\label{T:count}
Let $C_1$ and $C_2$ be two supersingular genus-$2$ curves over~$k$, with
supersingular invariants $I_1$ and $I_2$, respectively. If $I_1$ and $I_2$ are 
both nonzero, there are exactly sixty Richelot isogenies from $C_1$ to $C_2$, up
to isomorphism. If exactly one of $I_1$ and $I_2$ is zero, there are exactly
twelve Richelot isogenies from $C_1$ to $C_2$, up to isomorphism. If both $I_1$
and $I_2$ are zero, there are exactly four Richelot isogenies from $C_1$ 
to~$C_2$, up to isomorphism.
\end{theorem}

Supersingular curves have Jacobians isogenous to products of supersingular 
elliptic curves.  The following lemmas tell us a little more about the smallest
isogenies $E \times E \to J$ and $J \to E \times E$ and about polarizations on 
$E \times E$ obtained from~$J$.  To state the lemmas, we must first set some 
notation.

Let $E$ be the unique supersingular elliptic curve over~$k$.  The endomorphism 
ring of $E$ is a maximal order in the quaternion algebra $H$ over $\BQ$ that is
ramified at $2$ and infinity. We write $H = Q\langle i,j\rangle$ where 
$i^2 = j^2 = -1$ and $ij = -ji$, and we may take $\End E$ to be the order $\ord$ 
containing $i$, $j$, and $(1+i+j+ij)/2$.  There is a unique two-sided prime 
$\frakp$ of $\ord$ lying above~$2$; the quotient $\ord/\frakp$ is isomorphic 
to~$\BF_4$.

The elliptic curve $E$ has a unique principal polarization 
$p\colon E \to \hat{E}.$ We let $P\colon E \times E \to \hat{E} \times \hat{E}$
be the product polarization $p \times p$.

Let $M$ be the endomorphism of $E \times E$ given by the matrix
\[
\left[
\begin{matrix} 
      2 &    1 + i \\  
  1 - i &    2 
\end{matrix}
\right]
\]
in $M_2(\ord)$, and let $M'$ be the endomorphism
\[
\left[
\begin{matrix} 
      2 &   -1 - i \\
 -1 + i &   2 
\end{matrix}
\right].
\]

\begin{lemma}
\label{L:one}
Let $J$ be a $2$-dimensional supersingular Jacobian over $k$ and let $\psi_0$ be
a $2$-isogeny from $J$ to $E \times E$.  Then every $2$-isogeny
$\psi\colon J \to E \times E$ is of the form $\alpha \circ \psi_0$ for an 
automorphism $\alpha$ of $E \times E$.
\end{lemma}

\begin{proof}
As we noted above, $J$ has only one subgroup scheme of rank~$2$, so $\psi$ and 
$\psi_0$ have the same kernel. Therefore $\psi$ factors through~$\psi_0$.
\end{proof}

\begin{lemma}
\label{L:two}
Let $J$ be a $2$-dimensional supersingular Jacobian over $k$ and let $\varphi_0$
be a $2$-isogeny from $E \times E$ to~$J$.  Then every $2$-isogeny
$\varphi\colon E \times E \to J$ is of the form  $\varphi_0 \circ \alpha$ for an
automorphism $\alpha$ of $E \times E$.
\end{lemma}

\begin{proof}
The $2$-isogenies from $E \times E$ to $J$ correspond by duality to the
$2$-isogenies from $\hat{J}$ to $\hat{E} \times\hat{E}$.  Since $\hat{J}\cong J$
and $\hat{E} \times\hat{E} \cong E \times E$, Lemma~\ref{L:two} is just a 
restatement of Lemma~\ref{L:one}.
\end{proof}

\begin{lemma}
\label{L:three}
Let $J$ be a $2$-dimensional supersingular Jacobian over~$k$. There is a
degree-$2$ isogeny $\varphi\colon E \times E \to J$ that pulls the canonical
polarization of $J$ back to $P \circ M$, and there is a degree-$2$ isogeny 
$\varphi': E \times E \to J$ that pulls the canonical polarization back to
$P \circ M'$.
\end{lemma}

\begin{proof}
As we noted at the end of Section~\ref{S:review}, every supersingular Jacobian
is the target of a $2$-isogeny from $E\times E$. Let $\varphi$ be any such
isogeny. Then $\varphi$ pulls back the canonical polarization of $J$ to
\emph{some} degree-$4$ polarization of $E \times E$ with kernel 
$\alphatwo \times \alphatwo$.  But from~\cite{IbukiyamaKatsuraOort1986} we know
that the number of isomorphism classes of such polarizations is equal to the
class number $H_2(1,2)$ of the nonprincipal genus in $H \times H$, and this
class number is~$1$.  So we can modify $\varphi$ by automorphisms of
$E \times E$ in order to get the pullback of the canonical polarization to be
$P \circ M$ or $P \circ M'$.
\end{proof}

\section{Proof of the theorem}
\label{S:proof}
Let $(J_1, \lambda_1)$ and $(J_2, \lambda_2)$ be the canonically polarized 
Jacobians of $C_1$ and~$C_2$, respectively. From Lemma~\ref{L:three} we know
that there is a degree-$2$ isogeny $\varphi_2\colon E \times E \to J_2$ that
pulls the principal polarization $\lambda_2$ back to the degree-$4$ polarization 
$P \circ M$, and there is a degree-$2$ isogeny 
${\varphi_1\colon E \times E \to J_1}$ that pulls the principal polarization 
$\lambda_1$ back to the degree-$4$ polarization $P \circ M'$.

Let $\psi_1$ be the composite isogeny
\[
\xymatrix{
J_1\ar[rr]^{\lambda_1} && 
     \hat{J_1}\ar[rr]^{\hat{\varphi_1}} &&
          \hat{E}\times\hat{E}\ar[rr]^{P^{-1}} &&
               E\times E.
}
\]

\begin{claim}
\label{Claim}
Every Richelot isogeny from $C_1$ to $C_2$ can be written as a composition 
$\varphi_2\circ \alpha\circ \psi_1$ for some automorphism $\alpha$ of the 
polarized variety $(E \times E, P\circ M)$, and every such composition is a 
Richelot isogeny. 
\end{claim}

\begin{proof} 
Let us begin by showing that the diagram
\begin{equation}
\label{EQ:Diagram2}
\begin{gathered}
\xymatrix{
J_1       \ar[rr]^{2\lambda_1}\ar[d]_{\psi_1} && \hat{J_1} \\
E\times E \ar[rr]^{P\circ M}                  && \hat{E}\times\hat{E} \ar[u]_{\hat{\psi_1}} 
}
\end{gathered}
\end{equation}
is commutative.  Using the fact that 
$\psi_1 = P^{-1}\circ \hat{\varphi_1}\circ \lambda_1$,
we check that
\begin{align*}
 \hat{\psi_1}\circ P\circ M\circ \psi_1 
       & = \hat{\lambda_1}\circ \varphi_1\circ \hat{P^{-1}}\circ P\circ M\circ P^{-1}\circ \hat{\varphi_1}\circ \lambda_1\\
       & = \hat{\lambda_1}\circ \varphi_1\circ M\circ P^{-1}\circ \hat{\varphi_1}\circ \lambda_1 \\
       & \blankrel \phantom{2\circ \lambda_1} \text{\qquad[because $\hat{P} = P$, since $P$ is a polarization]}\\
       & = \hat{\lambda_1}\circ \varphi_1\circ M\circ P^{-1}\circ \hat{\varphi_1}\circ \lambda_1\circ \varphi_1\circ \varphi_1^{-1}\\
       & = \hat{\lambda_1}\circ \varphi_1\circ M\circ P^{-1}\circ P\circ M'\circ \varphi_1^{-1} \\
       & \blankrel \phantom{2\circ \lambda_1} \text{\qquad[because $\varphi_1$ pulls $\lambda_1$ back to $P\circ M'$]}\\
       & = \hat{\lambda_1}\circ \varphi_1\circ M\circ M'\circ \varphi_1^{-1}\\
       & = \hat{\lambda_1}\circ \varphi_1\circ 2\circ \varphi_1^{-1} \\
       & \blankrel \phantom{2\circ \lambda_1} \text{\qquad[because $M\circ M'$ is multiplication by $2$]}\\
       & = 2\circ \hat{\lambda_1}\\
       & = 2\circ \lambda_1 \text{\qquad[because $\lambda_1 = \hat{\lambda_1}$].}
\end{align*}
Also note that the diagram
\begin{equation}
\label{EQ:Diagram3}
\begin{gathered}
\xymatrix{
E\times E  \ar[rr]^{P\circ M} \ar[d]_{\varphi_2}  && \hat{E}\times\hat{E} \\
J_2        \ar[rr]_{\lambda_2}                    && \hat{J_2} \ar[u]_{\hat{\varphi_2}}
}
\end{gathered}
\end{equation}
is commutative, by the definition of~$\varphi_2$.

Now suppose that $\alpha$ is an automorphism of $(E \times E, P\circ M)$. This 
means that the diagram
\begin{equation}
\label{EQ:Diagram4}
\begin{gathered}
\xymatrix{
 E\times E \ar[rr]^{P\circ M} \ar[d]_{\alpha} &&  \hat{E}\times \hat{E} \\
 E\times E \ar[rr]_{P\circ M}                 &&  \hat{E}\times \hat{E} \ar[u]_{\hat{\alpha}}
}
\end{gathered}
\end{equation}
is commutative. Stacking Diagrams~\eqref{EQ:Diagram2}, \eqref{EQ:Diagram4}, 
and~\eqref{EQ:Diagram3} on top of one another, we get 
Diagram~\eqref{EQ:Diagram1}. Thus, given~$\alpha$, we get a Richelot isogeny 
from $C_1$ to~$C_2$.

Now suppose we are given a Richelot isogeny $\varphi$ from $C_1$ to~$C_2$. The 
kernel of $\varphi$ is a local-local subgroup scheme of~$J_1[2]$. This kernel 
must contain a copy of~$\alphatwo$, and there is a unique $\alphatwo$ in 
$J_1[2]$. So the isogeny $\varphi\colon J_1 \to J_2$ must factor 
through~$\psi_1$, say $\varphi = \gamma\circ\psi_1$ for an isogeny 
$\gamma:E\times E\to J_2$, and it follows that the middle arrow of the following
diagram is $P\circ M$:
\[
\xymatrix{
 J_1       \ar[rr]^{2\lambda_1} \ar[d]_{\psi_1}   && \hat{J_1} \\
 E\times E \ar[rr]^{P\circ M}   \ar[d]_{\gamma}   && \hat{E}\times \hat{E} \ar[u]_{\hat{\psi_1}}\\
 J_2       \ar[rr]^{\lambda_2}                    && \hat{J_2}\rlap{.}     \ar[u]_{\hat{\gamma}}
}
\]
Lemma~\ref{L:two} shows that $\gamma$ is $\varphi_2\circ \alpha$ for some 
automorphism $\alpha$ of $E\times E$, so we can further expand the diagram to 
get
\[
\xymatrix{
 J_1       \ar[rr]^{2\lambda_1} \ar[d]_{\psi_1}      && \hat{J_1} \\
 E\times E \ar[rr]^{P\circ M}   \ar[d]_{\alpha}      && \hat{E}\times \hat{E} \ar[u]_{\hat{\psi_1}}\\
 E\times E \ar[rr]^{P\circ M}   \ar[d]_{\varphi_2}   && \hat{E}\times \hat{E} \ar[u]_{\hat{\alpha}}\\
 J_2       \ar[rr]^{\lambda_2}                       && \hat{J_2}\rlap{.}     \ar[u]_{\hat{\varphi_2}}
}
\]
The middle part of the diagram shows that $\alpha$ is an automorphism of the 
principally-polarized variety ${(E\times E, P\circ M)}$.  This proves the claim.
\end{proof}

Every automorphism $\alpha$ of $E\times E$ can be written as an invertible 
matrix $A$ in~$M_2(\ord)$.  The automorphism $\alpha$ respects the polarization
$P\circ M$ if and only if we have $M = A^* M A$, where $A^*$ is the conjugate
transpose of~$A$.  (This is because the Rosati involution on $\End E\times E$
determined by the product polarization $P$ is equal to the conjugate transpose.)
With a computer algebra system it is not hard to compute\footnote{
     See Verification 4.1a in the Magma file.
} that there are $1920$ invertible matrices $A$ in $M_2(\ord)$ that satisfy this
condition.  Thus, $\#\Aut(E\times E,P\circ M) = 1920$.

Suppose $\beta$ is an automorphism of $(J_1, \lambda_1)$. Then $\beta$ must take
the unique copy of $\alphatwo$ in $J_1$ to itself, so $\beta$ gives an
automorphism of the polarized variety $(E\times E, P\circ M)$. This gives us an
(injective) homomorphism 
\[
\Aut(J_1, \lambda_1) \to \Aut(E\times E, P\circ M)
\]
that we denote by~$F_1$.

Now suppose $A$ is an automorphism of $(E\times E, P\circ M)$. If $A$ takes
$\ker \varphi_2$ to itself, then $A$ descends to give an automorphism of 
$(J_2, \lambda_2)$. Lemmas~1.4 and~1.5 of~\cite{KatsuraOort1987} show that every 
automorphism of $(J_2, \lambda_2)$ comes from a unique such~$A$. Thus, we get an
(injective) homomorphism
\[
\Aut(J_2, \lambda_2) \to \Aut(E\times E, P\circ M)
\]
that we denote by~$F_2$.

Claim~\ref{Claim} leads us to this key observation: The set of Richelot 
isogenies from $C_1$ to $C_2$ (up to isomorphism) is in bijection with the 
orbits of $\Aut(E\times E, P\circ M)$ under the combined actions of 
$\image F_1$ on the right and $\image F_2$ on the left. To make use of this 
observation we have to say a little more about $F_1$ and~$F_2$. First we look
at~$F_2$.

Suppose $A$ is an automorphism of $E\times E$. How does $A$ permute the various
copies of $\alphatwo$ sitting inside $E\times E$? Recall that the copies of 
$\alphatwo$ correspond to elements $[i\col j]$ of~$\PP^1(k)$. Let $B$ in 
$\GL_2 \BF_4$ be the reduction of $A$ modulo the two-sided prime $\frakp$
over~$2$. Using Dieudonn\'e modules, it is not hard to show that $A$ takes 
$[i\col j](\alphatwo)$ to $[i'\col j'](\alphatwo)$, where $[i'\col j']$ is 
$B[i\col j]$, under the natural action of $\GL_2$ on~$\PP^1$. (A description of
the Dieudonn\'e module of $E$ and of the action of $\End E$ on this module is 
given in Section 4.2 of~\cite{Waterhouse1969}.)  

By explicit calculation\footnote{
     See Verification 4.1b in the Magma file.
} we find that the image of $\Aut(E\times E, P\circ M)$ in $\GL_2 \BF_4$ is 
$\SL_2 \BF_4$. Let $G$ be the kernel of the reduction map from 
$\Aut(E\times E, P\circ M)$ to $\SL_2 \BF_4$, so that $\#G = 32$. We see that
the image of $F_2$ contains~$G$, and is equal to $G$ if $C_2$ is not the special
curve $y^2 + y = x^5$.

We can also say a little more about~$F_1$.  The Frobenius is a Richelot isogeny
from $C_1$ to $C_1'$ (where $C_1'$ denotes the curve defined by the same
equations as~$C_1$, but with all coefficients squared), so we can always draw a
diagram
\[
\xymatrix{
 J_1       \ar[rr]^{ 2\lambda_1} \ar[d]_{\psi_1}   && \hat{J_1} \\
 E\times E \ar[rr]^{P\circ M}    \ar[d]            && \hat{E}\times \hat{E} \ar[u]_{\hat{\psi_1}} \\
 J_1'      \ar[rr]^{ \lambda_1'}                   && \hat{J_1'} \ar[u]
}
\]
where the map $J_1 \to  J_1'$ on the left is Frobenius.  It follows that the 
image of $F_1$ contains~$G$, and is equal to $G$ unless $C_1$ is the special 
curve $y^2 + y = x^5$.

Now we use the key observation above to complete the proof of the theorem. If 
neither $C_1$ nor $C_2$ is the special curve, then the images of $F_1$ and 
$F_2$ are both equal to the normal subgroup $G$ of $\Aut(E\times E, P\circ M)$,
so the double cosets
\[
  G \,\backslash \Aut(E\times E, P\circ M) / G
\]
are in bijection with $\SL_2 \BF_4$, a group with $60$ elements.

If exactly one of $C_1$ or $C_2$ is isomorphic to the special curve, then the 
Richelot isogenies are in bijection with the cosets of $\SL_2 \BF_4$ by a 
non-normal subgroup of order~$5$. There are $12$ orbits.

If $C_1$ and $C_2$ are both isomorphic to the special curve, then the Richelot
isogenies are in bijection with the double cosets of $\SL_2 \BF_4$ by a 
non-normal subgroup of order $5$ on the left, and a (possibly different) such 
subgroup on the right. By direct calculation, we find that no matter what the 
subgroups, there are $4$ orbits.

This proves the theorem.
\qed

\section{Explicit constructions for purely inseparable Richelot isogenies}
\label{S:explicit}
Theorem~\ref{T:main} tells us how many Richelot isogenies there are between
two supersingular genus-$2$ curves over an algebraically closed field $k$
of characteristic~$2$. In this section we produce explicit constructions that
exhibit all of these isogenies.

In Sections~\ref{SS:dihedral} and~\ref{SS:degenerate} we present two
constructions of Richelot isogenies between supersingular genus-$2$ curves. In
Section~\ref{SS:FV} we discuss the Frobenius and Verschiebung isogenies, and we
prove that for supersingular Jacobians they are never isomorphic to one another.
In Section~\ref{SS:parameters} we show how to parametrize the Richelot isogenies
discussed in Section~\ref{SS:dihedral}, and in Section~\ref{SS:counting} we
prove that the constructions we have presented account for all of the Richelot
isogenies between two supersingular Jacobians.

\subsection{The dihedral construction}
\label{SS:dihedral}
Let $D$ be a supersingular hyperelliptic curve of genus~$4$ over~$k$, and
suppose $\alpha$ and $\beta$ are involutions of $D$ that anti-commute --- that
is, suppose $\alpha\beta = \iota\beta\alpha$, where $\iota$ is the 
hyperelliptic involution of~$D$. 

Let $G$ be the subgroup of the automorphism group $\Aut D$ of $D$ generated
by the involutions $\alpha$ and~$\beta$. The anti-commutation of $\alpha$ and
$\beta$ shows that $G$ is a dihedral group of order~$8$, so we can draw a 
diagram of the intermediate curves in the Galois cover $D \to D/G\cong \PP^1$
as follows:
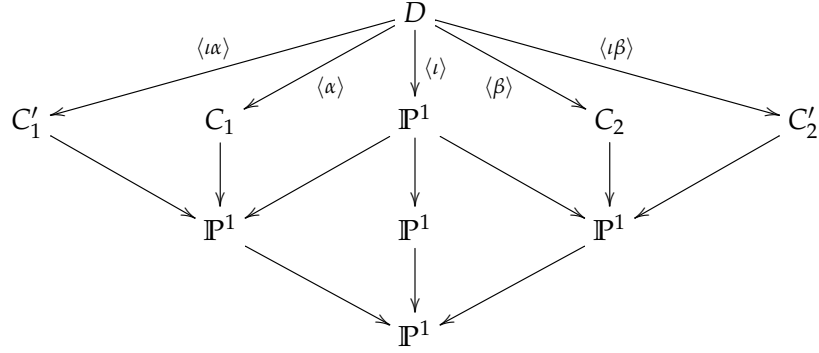
\begin{equation}
\label{EQ:dihedral}
\begin{gathered}
\xymatrix{
&& && D\ar[d]^{\langle\iota\rangle} && && \\
C_1'\ar[drr]\ar@{<-}[urrrr]^{\langle\iota\alpha\rangle}
&& C_1\ar[d]\ar@{<-}[urr]_{\langle\alpha\rangle}
&& \PP^1\ar[dll]\ar[drr]
&& C_2\ar[d]\ar@{<-}[ull]^{\langle\beta\rangle}
&& C_2'\ar[dll]\ar@{<-}[ullll]_{\langle\iota\beta\rangle} \\
&& \PP^1\ar[drr] && \PP^1\ar@{<-}[u] && \PP^1\ar[dll] && \\
&& && \PP^1\rlap{.}\ar@{<-}[u] && && \\
}
\end{gathered}
\end{equation}
Here the curves $C_1$ and $C_1'$ are isomorphic to one another because $\alpha$
and $\iota\alpha$ are conjugate to one another in~$G$; likewise, $C_2$ and 
$C_2'$ are isomorphic to one another.

From this large $D_8$ diagram we can extract the $V_4$ diagram from the upper
left:
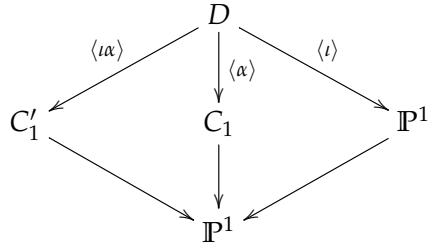
\begin{equation}
\label{EQ:V4left}
\begin{gathered}
\xymatrix{
&& D\ar[dll]_{\langle\iota\alpha\rangle}\ar[drr]^{\langle\iota\rangle} && \\
C_1'\ar[drr] && C_1\ar@{<-}[u]_{\langle\alpha\rangle} && \PP^1\ar[dll] \\
&& \PP^1\rlap{.}\ar@{<-}[u] &&\\
}
\end{gathered}
\end{equation}

By a result of Kani and Rosen~\cite[Theorem~C, p.~309]{KaniRosen1989}, this
$V_4$ diagram shows that the Jacobian of $D$ decomposes (up to isogeny) as the
sum of the Jacobians of $C_1$ and~$C_1'$, so that in particular the genus of
$C_1$ must be~$2$. Likewise, the genus of $C_2$ must be~$2$. Also, since $D$ is
supersingular, so must be $C_1$ and~$C_2$.

Let $\varphi_1$ be the natural degree-$2$ map from $D$ 
to~$C_1 \colonequals  D/\langle\alpha\rangle$ and let $\varphi_2$ be the natural
map from $D$ to~$C_2 \colonequals D/\langle\beta\rangle$. Each of the curves
$C_1$, $C_2$, and $D$ has a unique Weierstrass point, and each map $\varphi_i$
takes the Weierstrass point of $D$ to that of~$C_i$. Let $J$ be the Jacobian 
of~$D$, and for each~$i$ let $J_i$ be the Jacobian of~$C_i$.

\begin{proposition}
\label{P:dihedral}
The map $\varphi_{2*}\varphi_1^*\colon J_1\to J_2$ is a Richelot isogeny with 
dual isogeny $\varphi_{1*}\varphi_2^*$, and the curve $D$ and the involutions 
$\alpha$ and $\beta$ can be recovered \textup{(}up to isomorphism\textup{)} from
this isogeny.
\end{proposition}

\begin{proof}
Note that the homomorphism $\varphi_1^*\varphi_{1*}$ of $J$ is simply
$1 + \alpha^*$, and likewise $\varphi_2^*\varphi_{2*} = 1 + \beta^*$. If we let
$J_\alpha$ be the subvariety of $J$ where $\alpha^*$ acts trivially, then we
have a diagram
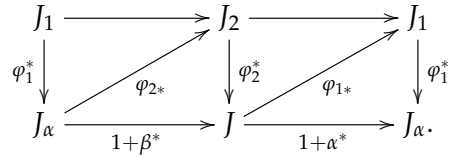
\[
\xymatrix{
J_1\ar[rr]\ar[d]_{\varphi_1^*} 
         && J_2\ar[rr]\ar[d]^{\varphi_2^*} 
                 && J_1\ar[d]^{\varphi_1^*}\phantom{.} \\
J_\alpha\ar[rr]_{1 + \beta^*}\ar[urr]_{\varphi_{2*}} 
         && J\ar[rr]_{1 + \alpha^*}\ar[urr]_{\varphi_{1*}} 
                 && J_\alpha.\\
}
\]
For every point $P\in J_\alpha(k)$ we have
\begin{align*}
(1 + \alpha^*)(1 + \beta^*)(P) 
 &= (1 + \alpha^* + \beta^* + \alpha^*\beta^*)(P)  \\
 &= (1 + \alpha^* + \beta^* - \beta^*\alpha^*)(P)  \\
 &= 2P + \beta^*(P) - \beta^*(P)\\
 &= 2P,
\end{align*}
so the composite map $J_\alpha\to J\to J_\alpha$ on the bottom of the diagram
is multiplication by~$2$. It follows that the composite map $J_1\to J_2\to J_1$
at the top of the diagram is also multiplication by~$2$.

Let $\lambda$, $\lambda_1$, and $\lambda_2$ be the canonical principal
polarizations of $J$, $J_1$, and~$J_2$, respectively. Lemma 4.4~(p.~186) 
of~\cite{HoweLauter2012} shows that the homomorphisms\footnote{
    The cited lemma mistakenly refers to ``isogenies'' when it should refer
    to ``homomorphisms.''}
$\varphi_{1*}$ and $\varphi_1^*$ are dual to one another, in the sense that
\[
\hat{\varphi_{1*}} = \lambda \varphi_1^* \lambda_1^{-1}
\quad\text{and}\quad
\hat{\varphi_1^*} = \lambda_1 \varphi_{1*} \lambda^{-1},
\]
and the analogous equalities hold for $\varphi_{2*}$ and $\varphi_2^*$. In the
preceding paragraph we showed that 
${(\varphi_{1*}\varphi_2^*) (\varphi_{2*}\varphi_1^*) = 2}$ on~$J_1$, and using
the duality we just mentioned we find 
\[
(\lambda_1^{-1} \hat{\varphi_1^*} \lambda 
\lambda^{-1} \hat{\varphi_{2*}} \lambda_2) (\varphi_{2*}\varphi_1^*) = 2,
\]
so that 
\[
(\lambda_1^{-1} \hat{\varphi_{2*}\varphi_1^*} \lambda_2) 
(\varphi_{2*}\varphi_1^*) = 2,
\]
and therefore 
\[
\hat{\varphi_{2*}\varphi_1^*} \lambda_2 (\varphi_{2*}\varphi_1^*) = 2\lambda_1.
\]
This is exactly what is shown in Diagram~\eqref{EQ:Diagram1}, so 
$\varphi_{2*}\varphi_1^*$ is a Richelot isogeny from $(J_1,\lambda_1)$ to
$(J_2,\lambda_2)$.

Similarly, $\varphi_{1*}\varphi_2^*$ is a Richelot isogeny from 
$(J_2,\lambda_2)$ to $(J_1,\lambda_1)$, and the two isogenies are dual to one 
another because their compositions in both orders are equal to multiplication
by~$2$.

Now we turn to the final statement of the proposition.

Suppose we are given an isogeny $\psi\colon J_1\to J_2$. Every such isogeny
comes from a correspondence on $C_1\times C_2$, that is, a divisor on 
$C_1\times C_2$ that does not consist solely of horizontal and vertical
components. We will construct a \emph{particular} correspondence that
represents~$\psi$.

Let $W_1$ and $W_2$ be the Weierstrass points on $C_1$ and~$C_2$, respectively,
and let $P\in C_1(k)$ be an arbitrary point with $P\ne W_1$. The isogeny $\psi$
takes the class of the degree-$0$ divisor $P - W_1$ on $C_1$ to a point on the 
Jacobian of~$C_2$, and every such point other than the identity can be 
represented as the class of a divisor $Q + R - 2W_2$ in a unique way.\footnote{
    The identity has an infinite number of such representations; it is
    equal to the class of $Q + \iota(Q) - 2 W_2$ for every point~$Q$.}
For every $P$ such that $[P-W_1]$ is not in the kernel of $\psi$ we let 
$\psibar(P)$ be the set $\{Q,R\}$, where $Q$ and $R$ are the points of $C_2$ 
such that $\psi(P - W_1) = [Q + R - 2W_2]$; for $P$ such that $[P-W_1]$ is in
the kernel of $\psi$ we take $\psibar(P)$ to be the empty set. Let  $\calW_0$
be the divisor on $C_1\times C_2$ consisting of the Zariski closure of the union
over all $P$ the sets $\{(P,Q) \colon Q\in \psibar(P)\}$. If there is a $P$ such
that $\#\psibar(P) = 2$, we take $\calW = \calW_0$; if there is no such~$P$, 
then we take $\calW = 2\calW_0$. Clearly the isogeny defined by $\calW$ is equal
to~$\psi$.

Note that the divisor $\calW$ is either an irreducible curve, the union of
two distinct irreducible curves, or twice a single irreducible  curve.

What does this construction produce when $\psi$ is the isogeny
$\varphi_{2*}\varphi_1^*$, as in the statement of the theorem? Let $W$ be the
Weierstrass point on~$D$, and for a given $P\in C_1(k)$ let $Q'$ and $R'$ be the
two points of $D(k)$ that map to~$P$. The pullback (via $\varphi_1$) of the
divisor $P - W_1$ is equal to $Q' + R' - 2W$, and the push-forward (via
$\varphi_2$) of this divisor is $\varphi_2(Q') + \varphi_2(R') - 2 W_2$. Thus,
the divisor $\calW$ on $C_1\times C_2$ is nothing other than the image of $D$
under the map $\varphi_1\times\varphi_2$, so $\calW$ is also a curve. But since
the degree-$2$ maps $\varphi_1\colon D\to C_1$ and $\varphi_2\colon D\to C_2$
factor through~$\calW$, we see that either $D$ is birationally equivalent to
$\calW$ and the projection maps from $\calW$ to $C_1$ and $C_2$ have degree~$2$,
or $\calW$ is birationally equivalent to both $C_1$ and $C_2$ and the projection
maps have degree~$1$. The latter possibility is inconsistent with 
$\varphi_{2*}\varphi_1^*$ having degree~$4$, so the map
$\varphi_1\times\varphi_2\colon D\to C_1\times C_2$ gives a birational
equivalence between $D$ and~$\calW$. Under this equivalence, the projection maps
from $\calW$ to $C_1$ and $C_2$ correspond to $\varphi_1$ and~$\varphi_2$. 

Thus, $D$ and the double covers $\varphi_1$ and $\varphi_2$ can be recovered, up
to isomorphism, from the Richelot isogeny $\varphi_{2*}\varphi_1^*$, and the 
involutions $\alpha$ and $\beta$ are determined by $\varphi_1$ and~$\varphi_2$.
This proves the proposition.
\end{proof}

\begin{remark}
\label{R:Dab}
Note that the isomorphism class of Diagram~\eqref{EQ:dihedral} is completely 
determined by the curve $D$ and the two pairs of involutions 
$\{\alpha, \iota \alpha\}$ and $\{\beta, \iota \beta\}$. Furthermore, the
pair $\{\alpha, \iota \alpha\}$  is determined by the involution of the top
$\PP^1$ in Diagram~\ref{EQ:dihedral} that fixes the left $\PP^1$, and the
pair $\{\beta, \iota \beta\}$  is determined by the involution of the top
$\PP^1$ that fixes the right~$\PP^1$.
\end{remark}

\subsection{The degenerate construction}
\label{SS:degenerate}
Suppose $C_1$ is a supersingular genus-$2$ curve over~$k$, and let $\gamma$ be
an automorphism of $C_1$ such that $\gamma^2$ is the hyperelliptic 
involution~$\iota$. Associated to $\gamma$ we have the pull-back automorphism
$\gamma^*$ of the polarized Jacobian $J_1$ of~$C_1$.

\begin{proposition}
\label{P:degenerate}
The map $1 + \gamma^*\colon J_1\to J_1$ is a Richelot isogeny with dual isogeny
$1 - \gamma^*$, and the automorphism $\gamma$ can be recovered 
\textup{(}up to conjugation in the automorphism group of~$J_1$\textup{)} from
this isogeny. This isogeny cannot be obtained from the construction of 
Proposition~\textup{\ref{P:dihedral}}.
\end{proposition}

\begin{proof}
Let $\gamma_*\in \Aut J_1$ be the push-forward of~$\gamma$. The composition 
$\gamma_*\gamma^*$ is simply multiplication by the degree of~$\gamma$, which
is~$1$; furthermore, since $\iota^* = -1$ we have $(\gamma^*)^2 = -1$, so that
$\gamma_* = -\gamma^*$. Thus we have a commutative diagram
\[
\xymatrix{
J_1\ar[rr]^{2}\ar[d]_{1 + \gamma^*} && {J_1}                \\
J_1\ar[rr]_{1}^{\sim} && {J_1}\ar[u]_{1 + \gamma_*} \rlap{.}
}
\]
From~\cite[Lemma~4.4, p.~186]{HoweLauter2012} we know that 
$\gamma_* = \lambda_1^{-1} \hat{\gamma^*} \lambda_1$, where $\lambda_1$ is the
canonical principal polarization of~$J_1$ and where $\hat{\gamma^*}$ is the 
dual isogeny of~$\gamma^*$. This means that we can extend the preceding diagram
to get
\[
\xymatrix{
J_1\ar[rr]^{2}\ar[d]_{1 + \gamma^*} && {J_1}\ar[rr]^{\lambda_1}                      && \hat{J_1}\\
J_1\ar[rr]_{1}^{\sim}               && {J_1}\ar[u]_{1 + \gamma_*}\ar[rr]_{\lambda_1} && \hat{J_1}\ar[u]_{1 + \hat{\gamma^*}} \rlap{.}
}
\]
The outer portion of this diagram is precisely Diagram~\eqref{EQ:Diagram1}, so we
find that $1 + \gamma^*$ is a Richelot isogeny from $(J_1,\lambda_1)$ to itself.
The dual isogeny is clearly $1 - \gamma^*$. Also, from the isogeny 
$1 + \gamma^*$ we can recover the automorphism $\gamma^*$ (up to conjugation),
and Torelli's theorem tells us that this specifies $\gamma$ (up to conjugation).

To show that this isogeny cannot be produced from the construction of the 
preceding section, we repeat the construction of the correspondence
$\calW\subset C_1\times C_1$ given in the proof of Proposition~\ref{P:dihedral}.
If $W_1$ is the Weierstrass point of $C_1$ and $P$ is an arbitrary point of 
$C_1$ such that $[P - W_1]$ is not in the kernel of $1 + \gamma^*$, we find that
\[
(1 + \gamma^*)([P-W_1]) = [P + \gamma^{-1}(P) - 2 W_1],
\]
so the divisor $\calW$ is the union of the graph of the identity of $C_1$ and
the graph of the automorphism $\gamma^{-1}$ of~$C_1$. In particular, $\calW$ is
the union of two irreducible curves (each isomorphic to~$C_1$). For the Richelot
isogenies in Proposition~\ref{P:dihedral}, the divisor $\calW$ was an
irreducible curve. Thus, no isogeny can be produced by both of these 
constructions.
\end{proof}

As we have just proven, these degenerate isogenies do not come from the dihedral
construction described in the preceding section. However, if we generalize the
dihedral construction to allow genus-$4$ ``curves'' $D$ that are \emph{not}
irreducible, then the degenerate isogenies fit into the same framework. Let us
sketch here how this works.

Let $D$ be the disjoint union of two copies of~$C_1$, which we denote by writing
$D = C_1\coprod C_1$. An automorphism of $D$ can either swap the copies of $C_1$
or not. Given two automorphisms $f$ and $g$ of~$C_1$, we will denote by
$\left[\begin{smallmatrix}f&0\\0&g\end{smallmatrix}\right]$ the automorphism
that sends the first component to itself by $f$ and the second component to
itself by~$g$; we will denote by 
$\left[\begin{smallmatrix}0&g\\f&0\end{smallmatrix}\right]$ the automorphism
that sends the first component to the second via $f$ and the second component to
the first via~$g$.

Suppose $\gamma$ is an automorphism of $C_1$ with $\gamma^2 = \iota$. Let 
$\alpha= \left[\begin{smallmatrix}0&1\\1&0\end{smallmatrix}\right]$, let
$\beta = \left[\begin{smallmatrix}0&\gamma\\ \iota\gamma&0\end{smallmatrix}\right]$,
and let $I = \left[\begin{smallmatrix}\iota&0\\0&\iota\end{smallmatrix}\right]$.
Note that $I^*$ acts as $-1$ on the Jacobian $J = J_1\oplus J_1$ of~$D$, so $I$
plays the role of the hyperelliptic involution; it is central in the 
automorphism group of~$D$.

We see that $\alpha$ and $\beta$ are involutions that anti-commute with one
another, in the sense that  $\alpha\beta = I \beta\alpha$, and the subgroup $G$
of $\Aut D$ generated by them is dihedral of order~$8$. As in 
Section~\ref{SS:dihedral}, we get a Galois cover $D \to D/G\cong \PP^1$:
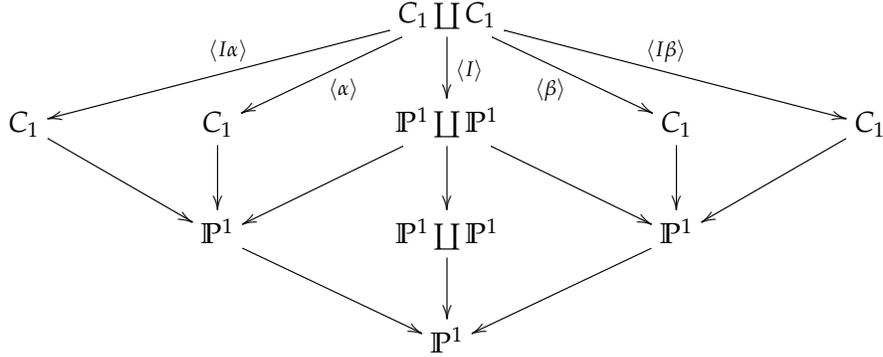
\[
\xymatrix{
&& && C_1\coprod C_1\ar[d]^{\langle I\rangle} && && \\
C_1\ar[drr]\ar@{<-}[urrrr]^{\langle I\alpha\rangle}
&& C_1\ar[d]\ar@{<-}[urr]_{\langle\alpha\rangle}
&& \PP^1 \coprod \PP^1 \ar[dll]\ar[drr]
&& C_1\ar[d]\ar@{<-}[ull]^{\langle\beta\rangle}
&& C_1\ar[dll]\ar@{<-}[ullll]_{\langle I\beta\rangle} \\
&& \PP^1\ar[drr] && \PP^1\coprod \PP^1 \ar@{<-}[u] && \PP^1\ar[dll] && \\
&& && \PP^1\rlap{.}\ar@{<-}[u] && && \\
}
\]
If we let $\Id$ denote the identity map on~$C_1$, then the four maps from 
$C_1\coprod C_1$ to $C_1$ in this diagram are, from left to right, isomorphic to
$(\iota,\iota)$, $(\Id, \Id)$, $(\iota\gamma,\Id)$, and $(\gamma,\iota)$. In the
notation of Section~\ref{SS:dihedral}, we have $\varphi_1 = (\Id,\Id)$ and 
$\varphi_2 = (\iota\gamma,\Id)$.

Now we can compute where the homomorphism 
${\varphi_{2*}\varphi_1^*\colon J_1\to J_1}$ sends the class of a divisor
$P - W_1$. We see that $\varphi_1^*(P)$ consists of two points: $P$ on one copy
of $C_1$ in $C_1\coprod C_1$, and $P$ on the other copy of~$C_1$. Applying 
$\varphi_2$ to this divisor gives $P + \iota\gamma(P)$. Similarly, we find that
$W_1$ gets sent to $2W_1$, so 
\[
\varphi_{2*}\varphi_1^*([P-W_1]) = [P + \iota\gamma(P) - 2 W_1].
\]
But $\gamma^*(P) = \iota\gamma(P)$, so we find that 
$1 + \gamma^* = \varphi_{2*}\varphi_1^*$, and the Richelot isogeny we produced
with the degenerate construction can be viewed as coming from a generalized
version of the dihedral construction.

\subsection{Frobenius and Verschiebung}
\label{SS:FV}
To be complete, we will make a few remarks on the Richelot isogenies given by 
Frobenius and Verschiebung. Let $C_1$ be an arbitrary supersingular genus-$2$ 
curve over~$k$, say given by an equation $y^2 = x^5 + Bx^3$, and let $C_2$ be 
the curve  $y^2 = x^5 + B^2x^3$. Let $(J_1,\lambda_1)$ and $(J_2,\lambda_2)$ be
the polarized Jacobians of these curves. The Frobenius morphism $\varphi$ from 
$C_1$ to $C_2$ takes a point $(x,y)$ to the point $(x^2,y^2)$, and the 
push-forward of this map is an isogeny $F\colon J_1\to J_2$ also called the 
Frobenius. The isogeny $V\colon J_2\to J_1$ given by 
$V= \lambda_1^{-1}\hat{F}\lambda_2$ is called the Verschiebung. We have a 
diagram
\[
\xymatrix{
J_1\ar[rr]^{2\lambda_1}\ar[d]_{F} && \hat{J_1}                 \\
J_2\ar[rr]^{\lambda_2}            && \hat{J_2}\ar[u]_{\hat{F}}
}
\]
that shows that the Frobenius is a Richelot isogeny and that its dual isogeny is
the Verschiebung.

If we apply the construction of the divisor $\calW$ from the proof of 
Proposition~\ref{P:dihedral} to the isogeny $\psi = F$, we find that 
$\psibar(P) = \{\varphi(P), W_2\}$ for every $P\neq W_1$ in~$C_1(k)$, so that 
$\calW\subset C_1\times C_2$ is the union of the graph of $\varphi$ and a 
``horizontal'' copy of~$C_1$.

Similarly, if we construct $\calW\subset C_2\times C_1$ from the Verschiebung
$V\colon J_2\to J_1$, we obtain the transpose of the graph of $\varphi$ together
with a ``horizontal'' copy of~$C_2$.

We see that neither Frobenius nor Verschiebung can be constructed from the
dihedral construction or from the degenerate construction from the preceding
sections.

There is one remaining point to consider: Are the Frobenius and Verschiebung
isogenies starting from a given Jacobian ever isomorphic to one another? The 
following theorem says that the answer is no.

\begin{theorem}
\label{T:FneqV}
Let $C$ be a supersingular genus-$2$ curve over $k$ with supersingular 
invariant~$I$, let $C^{(2)}$ and $C^{(1/2)}$ be the curves with invariants $I^2$
and~$I^{1/2}$, respectively, and let $J$, $J^{(2)}$, and~$J^{(1/2)}$ be the 
Jacobians of these curves. Then Frobenius isogeny $F\colon J\to J^{(2)}$ is 
never isomorphic to the Verschiebung isogeny $v\colon J\to J^{(1/2)}$.
\end{theorem}

\begin{proof}
Suppose, to get a contradiction, that $F$ and $V$ are isomorphic. This means 
that there are isomorphisms $a\colon J\to J$ and $b\colon J^{(2)} \to J^{(1/2)}$
of polarized Jacobians such that $V = b \circ F\circ a$. Let $F'$ be the 
Frobenius morphism from  $J^{(1/2)}$ to~$J$. By definition, $F'V$ is 
multiplication by $2$ on~$J$.

Torelli's theorem tells us that there are isomorphisms $\alpha\colon C\to C$ and
$\beta\colon C^{(2)}\to C^{(1/2)}$ such that $a = \alpha_*$ and $b = \beta_*$. 
Also, if we let $\varphi\colon C\to C^{(2)}$ and $\varphi'\colon C^{(1/2)}\to C$
be the Frobenius morphisms of curves, then $F = \varphi_*$ and 
$F' = \varphi'_*$.

Let $\infty$ be the infinite point on $C$ and let $P$ be any point on~$C$ other
than~$\infty$. Then in the Jacobian of $C$ we have
\begin{align*}
[2P - 2\infty] 
  &= 2[P-\infty] \\
  &= (F'\circ V)([P-\infty])\\
  &= (F'\circ b \circ F\circ a)([P-\infty])\\
  &= (\varphi'_*\circ\beta_*\circ\varphi_*\circ\alpha_*)([P - \infty])\\
  &= [(\varphi'\circ\beta\circ\varphi\circ\alpha)(P) - \infty]\\
  &= [Q - \infty]\\
  &= [Q + \infty - 2\infty]
\end{align*}
where $Q = (\varphi'\circ\beta\circ\varphi\circ\alpha)(P)$. But since
$P\neq \infty$, this equality contradicts the fact (noted above) that every 
nonzero point on $J$ has a unique representation as the class of a divisor of 
the form $[R + S - 2\infty]$. 
\end{proof}

\subsection{Parametrizing dihedral diagrams}
\label{SS:parameters}
The dihedral construction that we presented in Section~\ref{SS:dihedral} begins 
with a supersingular hyperelliptic curve of genus $4$ with two anti-commuting 
involutions, so we must ask: Do any such curves exist? In this section we
present a family of such curves and show that all such curves are members of the
family.

Throughout this section, $k$ will be an algebraically closed field of 
characteristic~$2$. Let $r$ and $s$ be two distinct nonzero elements of~$k$. We
produce a diagram like Diagram~\eqref{EQ:dihedral}, in which $D$ is 
supersingular of genus~$4$ and $D$ is Galois over the bottom~$\PP^1$, as 
follows:

Let $a$ and $c$ be elements of $k$ with $a^{40} = r$ and $c^{40} = s$. (We will
show that the isomorphism class of the diagram we construct does not depend on
the choice of $a$ and~$c$.) Set
\begin{align}
\label{EQ:B1} B_1 &= \frac{a^5 + c^{10} + 1}{a^4 c^2}\\
\label{EQ:B2} B_2 &= \frac{a^{10} + c^5 + 1}{a^2 c^4},
\end{align}
and let $b = c^2$ and $d = a^2$. Finally, let $e$ be a root of the polynomial
\begin{equation}
\label{EQ:additive}
p_{16} T^{16} + p_8 T^8 + p_4 T^4 + p_2 T^2 + p_1 T + p_0,
\end{equation}
where 
\begin{equation}
\label{EQ:pvalues}
\left\{\begin{aligned}
p_{16} &= a^8 c^8\\
p_8    &= (a^{10} + a^5 c^5 + c^{10} + 1)^2 \\
p_4    &= a^6 c^6 (a^5 + 1)^2 (c^5 + 1)^2\\
p_2    &= a^4 c^4 (a^{10} + a^5 c^5 + c^{10} + a^5 + c^5)^2\\
p_1    &= a^8 c^8 (a^5 + c^5)^2\\
p_0    &= a^7 c^7 (a^5 + 1) (c^5 + 1) (a^5 + c^5)^2.
\end{aligned}
\right.
\end{equation}

Let $t$ be a parameter for a copy of~$\PP^1$, and set
\begin{align*}
x &= (c t^2 + a^4 t + c e^2) / (a^5 + c^5)\\
u &= (a t^2 + c^4 t + a e^2) / (a^5 + c^5).
\end{align*}
Note that then we have\footnote{
     See Verification 5.4a in the Magma file.
}
\begin{equation}
\label{EQ:conic}
a^2 x^2 + c^4 x + c^2 u^2 + a^4 u + e^2 = 0.
\end{equation}
Finally, let $z = a^2 x^2 + c^4 x$. By~\eqref{EQ:conic}, we also have
$z = c^2 u^2 + a^4 u + e^2$. This gives us a $V_4$ diagram of copies of~$\PP^1$,
with $t$ being a parameter for the top copy, $x$ a parameter for the left copy,
$u$ a parameter for the right copy, and $z$ a parameter for the bottom copy. We 
give names to some of these covers as in the following diagram:
\begin{equation}
\begin{gathered}
\xymatrix{
&& \PP^1\ar[dll]_{\chi_1}\ar[drr]^{\chi_2} &&\\
\PP^1\ar[drr] && \PP^1\ar@{<-}[u] && \PP^1\ar[dll] \\
&& \PP^1\rlap{.}\ar@{<-}[u] &&\\
}
\end{gathered}
\end{equation}

Now consider the extension $\psi_1\colon C_1\to\PP^1$ of the leftmost $\PP^1$ 
defined by $y^2 + y = x^5 + B_1 x^3$ and the extension 
$\psi_2\colon C_2\to\PP^1$ of the rightmost $\PP^1$ defined by 
$v^2 + v = u^5 + B_2 u^3$. Each of these Artin--Schreier extensions gives rise 
to an Artin--Schreier extension of the top $\PP^1$, and we claim that these two
extensions are in fact isomorphic as Artin--Schreier extensions of~$\PP^1$; to
show this, we need only check that the element
\[
R\colonequals x^5 + B_1 x^3 + u^5 + B_2 u^3
\]
of $k(t)$ can be written $S^2 + S$ for some element $S$ of~$k(t)$.

We produce such an $S$ as follows: Let
\begin{align*}
 s_5 &= \frac{1}{(a^5+c^5)^2}\\[1ex]
 s_4 &= \frac{e}{(a^5+c^5)^2}\\[1ex]
 s_3 &= \frac{a^{10} + a^5 c^5 + a^5 + c^{10} + c^5}{a^2 c^2 (a^5 + c^5)^2}\\[1ex]
 s_2 &= \frac{e^8}{(a^5 + c^5)^2} 
       + \frac{e^4 (a^{10} + a^5 c^5 + c^{10} + 1)}{a^4 c^4 (a^5 + c^5)^2}
       + \frac{e^2 (a^5 + 1) (c^5 + 1)}{a c (a^5 + c^5)^2}\\[1ex]
 s_1 &= \frac{e^4 (a^5 + c^5 + 1)}{(a^5 + c^5)^2}
\end{align*}
and let $s_0$ be an element of $k$ such that 
\[
s_0^2 + s_0 
= \frac{e^{10}}{(a^5 + c^5)^4} 
      + \frac{e^6 (a^{10} + a^5 c^5 + c^{10} + 1)}{a^4 c^4 (a^5 + c^5)^2}.
\]
Let $S = s_5 t^5 + s_4 t^4 + s_3 t^3 + s_2 t^2 + s_1 t + s_0$. With the aid of a
computer algebra system, we check that then $R = S^2 + S$, as desired.\footnote{
     See Verification 5.4b in the Magma file.
} Thus, if we let $\psi\colon D\to\PP^1$ be the double cover of the top $\PP^1$
induced by $\psi_1$ and $\chi_1$, then by identifying $v$ with $y + S$ we can
also take $\psi\colon D\to \PP^1$ to be the double cover of the top $\PP^1$ 
induced by $\psi_2$ and $\chi_2$.

Since $D$ is obtained as the compositum of the two Artin--Schreier covers
$\psi_1$ and $\chi_1$ of the leftmost $\PP^1$, there is also a third curve
intermediate between $D$ and the leftmost $\PP^1$. We call this curve~$C_1'$,
and we let $\varphi_1\colon D\to C_1$ and $\varphi_1'\colon D\to C_1'$ be the
associated double covers. We define $\varphi_2\colon D\to C_2$ and 
$\varphi_2'\colon D\to C_2'$ analogously. At this point, we have the following
diagram:
\begin{equation}
\label{EQ:dihedral-rs}
\begin{gathered}
\xymatrix{
&& && D\ar[d]^{\psi} && && \\
C_1'\ar[drr]_{\psi_1'}\ar@{<-}[urrrr]^{\varphi_1'}
&& C_1\ar[d]^{\psi_1}\ar@{<-}[urr]_{\varphi_1}
&& \PP^1\ar[dll]^{\chi_1}\ar[drr]_{\chi_2}
&& C_2\ar[d]^{\psi_2}\ar@{<-}[ull]^{\varphi_2}
&& C_2'\ar[dll]^{\psi_2'}\ar@{<-}[ullll]_{\varphi_2'} \\
&& \PP^1\ar[drr] && \PP^1\ar@{<-}[u] && \PP^1\ar[dll] && \\
&& && \PP^1\rlap{.}\ar@{<-}[u] && && \\
}
\end{gathered}
\end{equation}

The involution $\alpha$ of $D$ that fixes $C_1$ is given by $t\to t + a^4/c$ and
$y\mapsto y$. The involution $\beta$ of $D$ that fixes $C_2$ is given by 
$t \to t + c^4/a$ and $v\mapsto v$, so that 
\[
\beta(y) = \beta(v + S(t)) = v + S\Bigl(t + \frac{c^4}{a}\Bigr)  = y + S(t) + S\Bigl(t + \frac{c^4}{a}\Bigr).
\]
We then check that
\begin{align*}
\alpha\beta\alpha\beta(y) 
 &= \alpha\beta\alpha\Bigl(y + S(t) + S\Bigl(t + \frac{c^4}{a}\Bigr)\Bigr)\\
 &= \alpha\beta\Bigl(y + S\Bigl(t+ \frac{a^4}{c}\Bigr) + S\Bigl(t + \frac{a^4}{c} + \frac{c^4}{a}\Bigr)\Bigr)\\
 &= \alpha\Bigl(y + S(t) + S\Bigl(t + \frac{c^4}{a}\Bigr) + S\Bigl(t + \frac{a^4}{c} + \frac{c^4}{a}\Bigr) + S\Bigl(t+ \frac{a^4}{c}\Bigr)\Bigr)\\
 &= y + S\Bigl(t+ \frac{a^4}{c}\Bigr) + S\Bigl(t + \frac{a^4}{c} + \frac{c^4}{a}\Bigr) + S\Bigl(t + \frac{c^4}{a}\Bigr) + S(t),
\end{align*}
and this last expression simplifies\footnote{
     See Verification 5.4c in the Magma file.
} to $y + 1$. Since $\alpha\beta\alpha\beta$ is easily seen to fix~$t$, we find
that $\alpha\beta\alpha\beta$ is in fact the hyperelliptic involution $\iota$ of
$D\to\PP^1$. This is enough to show that $\alpha$ and $\beta$ generate a
dihedral group of order~$8$.

Thus, Diagram~\eqref{EQ:dihedral-rs} is an example of 
Diagram~\eqref{EQ:dihedral}.

\begin{lemma}
\label{L:ace}
Up to isomorphism, Diagram~\eqref{EQ:dihedral-rs} is specified completely by the 
two equations
\begin{align*}
y^2 + y &= x^5 + B_1 x^3\\
0 &= a^2 x^2 + c^4 x + c^2 u^2 + a^4 u + e^2,
\end{align*}
where $B_1$ is given by~\eqref{EQ:B1}. 
\end{lemma}

\begin{proof}
This follows from Remark~\ref{R:Dab}, because the hyperelliptic curve $D$ is
specified by the two given equations, while the top $\PP^1$ in the diagram is
given by the conic defined by the equation in $x$ and~$u$, and the two relevant
involutions of the conic are $(x,u)\mapsto(x,u + a^4/c^2)$ and
$(x,u) \mapsto (x + c^4/a^2, u)$.
\end{proof}

\begin{lemma}
The isomorphism class of Diagram~\eqref{EQ:dihedral-rs} depends only on~$r$
and~$s$.
\end{lemma}

\begin{proof}
First we show that for a given choice of $a$ and $c$ with $a^{40} = r$ and
$c^{40} = s$, the isomorphism class of Diagram~\eqref{EQ:dihedral-rs} does not
depend on our choice of the root $e$ of~\eqref{EQ:additive}.

Let $e$ be a root of~\eqref{EQ:additive}. The diagram we obtain is then
specified up to isomorphism by the two equations in Lemma~\ref{L:ace}. Let $r$ 
be a root of the polynomial
\[
g\colonequals T^{16} + B_1^4 T^8 + B_1^2 T^2 + T,
\]
so that by Remark~\ref{R:alternateform} there is a polynomial $f\in k[x]$ such
that $(x,y)\mapsto (x+r,y+f)$ is an automorphism of the curve
$y^2 + y = x^5 + B_1 x^3$. Let $\eps\colonequals a r + c^2 \sqrt{r}$. Then the
map $(x,u,y)\mapsto (x+r,u,y+f)$ gives an isomorphism from the curve specified
by the equations in Lemma~\ref{L:ace} to the curve specified by the same
equations, but with $e$ replaced with $e +\eps$. Furthermore, this isomorphism 
identifies the two involutions $x\mapsto x + c^4/a^2$ and $u\mapsto u + a^4/c^2$
of the conic involving $e$ to the same involutions of the conic involving
$e + \eps$, so the diagram involving $e$ is isomorphic to the one involving 
$e + \eps$.

We check\footnote{
     See Verification 5.6a in the Magma file.
} that $\eps$ is a root of the polynomial 
\[
p_{16} T^{16} + p_8 T^8 + p_4 T^4 + p_2 T^2 + p_1 T
\]
obtained from \eqref{EQ:additive} by removing the constant term. It follows that
$e + \eps$ is a root of~\eqref{EQ:additive}. If we can show that every root 
of~\eqref{EQ:additive} can be obtained from a root of $g$ in this way, then we
will have shown that the isomorphism class of Diagram~\eqref{EQ:dihedral-rs}
does not depend on our choice of the root~$e$.

The $16$ roots of $g$ form an additive group that acts on the roots 
of~\eqref{EQ:additive} via $(r,e)\mapsto e + a r + c^2 \sqrt{r}.$ To show that 
the action is transitive we need only show that no nonzero root $r$ of $g$ 
satisfies $a r + c^2\sqrt{r} = 0$; that is, we need only show that $c^4/a^2$ is
not a root of~$g$. By direct calculation\footnote{
     See Verification 5.6b in the Magma file.
} we find that $g(c^4/a^2) = c^4 (a^5 + c^5)^4/a^{32}$, which is nonzero. 
Therefore, up to isomorphism, Diagram~\eqref{EQ:dihedral-rs} does not depend on
our choice of the root~$e$.

Now we move on to showing that the isomorphism class of the diagram
does not depend on our choice of the elements $a$ and $c$ with $a^{40} = r$
and $c^{40} = s$. Let $\zeta_1$ and $\zeta_2$
be two fifth roots of unity, and let us repeat the construction with
$a$ and $c$ replaced with $a'\colonequals a\zeta_1$ and
$c' \colonequals c\zeta_2$, respectively.

Going through the construction with these new values of $a$ and $c$, we find 
that the new values of several other constants are related to the old ones as
follows:
\begin{align*}
B_1'    &= B_1\zeta_1\zeta_2^3       & 
p_{16}' &= p_{16}\zeta_1^3\zeta_2^3  &  
p_4'    &= p_4 \zeta_1\zeta_2        & 
p_1'    &= p_1\zeta_1^3 \zeta_2^3    \\
B_2'    &= B_2\zeta_1^3\zeta_2       &
p_8'    &= p_8                       &
p_2'    &= p_2\zeta_1^4\zeta_2^4     &
p_0'    &= p_0 \zeta_1^2\zeta_2^2.
\end{align*}
Then one possible choice for $e'$ would be $e' \colonequals e\zeta_1^4\zeta_2^4$,
and since the choice of a root of~\eqref{EQ:additive} does not affect the
isomorphic class of the diagram, we may take this $e'$ as our choice.

Let us take the parameter $t'$ in the new construction to be 
$t\zeta_1^4\zeta_2^4$. Then we have $x' = x \zeta_1^3\zeta_2^4$ and
$u' = \zeta_1^4\zeta_2^3 u$, so that $(x')^5 + B_1'(x')^3 = x^5 + B_1 x^3$ and 
$(u')^5 + B_2'(u')^3 = u^5 + B_2 u^3$. Thus, by setting $y' = y$ and $v'=v$, we
obtain an isomorphism from the diagram obtained from $a$ and $c$ to the one
obtained from $a'$ and $c'$.
\end{proof}

\begin{lemma}
The invariants of the curves $C_1$ and $C_2$ in Diagram~\eqref{EQ:dihedral-rs}
are given by
\begin{align*}
I_1 &= \frac{(r + s^2 + 1)^5 }{r^4 s^2}\\
I_2 &= \frac{(r^2 + s + 1)^5 }{r^2 s^4}.
\end{align*}
\end{lemma}

\begin{proof}
This follows by raising both sides of~\eqref{EQ:B1} and~\eqref{EQ:B2} to the
$40$th power.
\end{proof}

We arrive at the main results of this section.

\begin{theorem}
\label{T:DiagramGeneral}
Let $C_1$ and $C_2$ be supersingular genus-$2$ curves with invariants $I_1$ 
and~$I_2$, respectively. Up to isomorphism, diagrams in the form of 
Diagram~\textup{\eqref{EQ:dihedral}} containing this $C_1$ and $C_2$ are in 
bijection with pairs $(r,s)$ of distinct nonzero elements of $k$ satisfying
\begin{align*}
(r + s^2 + 1)^5 &= r^4 s^2 I_1\\
(r^2 + s + 1)^5 &= r^2 s^4 I_2.
\end{align*}
\end{theorem}

\begin{proof}
Given a diagram of the form~\eqref{EQ:dihedral} containing the given $C_1$ 
and~$C_2$, we will show how to determine a unique pair $(r,s)$ that gives rise
to it via the construction described in this section.

Choose elements $B_1$ and $B_2$ of $k$ so that $B_1^{40} = I_1$ and 
$B_2^{40} = I_2$, and choose coordinates $x$ and $u$ for the leftmost and 
rightmost $\PP^1$s in the diagram, respectively, so that the extensions 
$C_1\to\PP^1$ and $C_2\to\PP^1$ are given by $y^2 + y = x^5 + B_1x^3$ and 
$v^2 + v = u^5 + B_2u^3$, respectively. 

We begin by focusing on a subdiagram, looking especially at the relationships
among the copies of $\PP^1$:
\begin{equation}
\label{EQ:V4bottomwithD}
\begin{gathered}
\xymatrix{
&& D\ar[d] &&\\
C_1\ar[d]\ar@{<-}[urr]&& \PP^1\ar[dll]\ar[drr] && C_2\ar[d]\ar@{<-}[ull]\\
\PP^1\ar[drr] && \PP^1\ar@{<-}[u] && \PP^1\ar[dll] \\
&& \PP^1\rlap{.}\ar@{<-}[u] &&\\
}
\end{gathered}
\end{equation}
Since the hyperelliptic curve $D$ in Diagram~\eqref{EQ:dihedral} is 
supersingular, the double cover $D\to\PP^1$ of the top $\PP^1$ is ramified at
only one point. The point $\infty$ in the left $\PP^1$ ramifies in 
$C_1\to\PP^1$, so if two points of the top $\PP^1$ mapped to $\infty$ in the
left~$\PP^1$, both of those points would ramify in the double cover $D\to\PP^1$,
a contradiction. Therefore $\infty$ ramifies in the top left cover 
$\PP^1\to\PP^1$, and similarly in the top right cover $\PP^1\to\PP^1$. It 
follows that $\infty$ in the left $\PP^1$ and $\infty$ in the right $\PP^1$ map
to the same point in the bottom~$\PP^1$, and this point in the bottom $\PP^1$ 
ramifies going up to the left and going up to the right. In fact, this point is
the only point of the bottom $\PP^1$ that ramifies going up to $D$, and it
ramifies totally.

We can choose a coordinate $z$ on the bottom $\PP^1$ so that this common image
point is~$\infty$. Then the bottom left map $\PP^1\to\PP^1$ is of the form 
$z = a^2 x^2 + b^2 x + e_1^2$ and the bottom right map $\PP^1\to\PP^1$ is 
$z = c^2 u^2 + d^2 u + e_2^2$. (We take the coefficients to be squares to avoid
having square roots appear in later formul\ae.) Setting $e = e_1 + e_2$, we find
that the $V_4$ diagram of copies of $\PP^1$ is specified by the relation
\begin{equation}
\label{EQ:conic2}
a^2 x^2 + b^2 x + c^2 u^2 + d^2 u + e^2 = 0,
\end{equation}
for some $a,b,c,d,e \in k$. We note that $a$, $b$, $c$, and~$d$ must all be 
nonzero in order for the bottom left and bottom right covers $\PP^1\to\PP^1$ to 
be separable maps of degree~$2$. The conic defined by this equation is the 
$\PP^1$ that appears in the top of the $V_4$ diagram, so the conic must be 
nonsingular, which means simply that $a d^2 + b^2 c \ne 0$.

The conic defined by~\eqref{EQ:conic2} does not change if we multiply all of the
coefficients by a nonzero scalar, and there is a unique scaling factor that will
result in the equality $d = a^2$. We scale the coefficients so that this
equality holds.

If we take $t = ax + cu$, then $t$ generates the function field of the
top~$\PP^1$, and we have\footnote{
     See Verification 5.8a in the Magma file.
}
\begin{align}
\label{EQ:x2} x &= (c t^2 + d^2 t + c e^2) / (a d^2 + b^2 c)\\
\label{EQ:u2} u &= (a t^2 + b^2 t + a e^2) / (a d^2 + b^2 c).
\end{align}

Recall that Diagram~\eqref{EQ:V4left} shows a $V_4$ extension extracted from 
Diagram~\eqref{EQ:dihedral}. The middle extension $C_1\to \PP^1$ in 
Diagram~\eqref{EQ:V4left} is the Artin--Schreier extension 
\[
y^2 + y = x^5 + B_1 x^3.
\]
The extension $\PP^1\to\PP^1$ on the right is given by~\eqref{EQ:conic2}, which
we can scale and rewrite as an Artin--Schreier equation:
\[
\left(\frac{c^2}{d^2} u\right)^2 + \left(\frac{c^2}{d^2} u\right)
  = \frac{a^2c^2}{d^4} x^2 + \frac{b^2c^2}{d^4} x + \frac{e^2c^2}{d^4}.
\]
Therefore the double cover $C_1'\to\PP^1$ is given by
\[
\left(y + \frac{c^2}{d^2} u\right)^2 + \left(y + \frac{c^2}{d^2} u\right)
  =  x^5 + B_1 x^3 + \frac{a^2c^2}{d^4} x^2 + \frac{b^2c^2}{d^4} x 
         + \frac{e^2c^2}{d^4}.
\]

Consider the involution $\beta$ of~$D$, which we know fixes $C_2$ and takes 
$C_1$ to~$C_1'$. On the $V_4$ diagram of copies of~$\PP^1$ at the bottom 
of~\eqref{EQ:V4bottomwithD}, the involution $\beta$ must then act trivially on
the rightmost $\PP^1$ and nontrivially on the leftmost~$\PP^1$. Since $\beta$
fixes $C_2$ we must have $\beta^* u = u$ and $\beta^* v = v$. Since $\beta$ acts
nontrivially on the leftmost $\PP^1$ but fixes the bottom~$\PP^1$, we must have
$\beta^* x = x + b^2/a^2$. And since $\beta$ takes $C_1$ to~$C_1'$, we must have
\[
\beta^* y = y +  \frac{c^2}{d^2} u + F
\]
for some $F\in k(x)$. Applying $\beta^*$ to $y^2 + y = x^5 + B_1 x^3$ we find
that
\[
\left(y +  \frac{c^2}{d^2} u + F\right)^2 
  + \left(y +  \frac{c^2}{d^2} u + F\right)
  = \left(x + \frac{b^2}{a^2}\right)^5 + B_1 \left(x + \frac{b^2}{a^2}\right)^3,
\]
which simplifies\footnote{
     See Verification 5.8b in the Magma file for this statement and the 
     following one.
} to
\begin{multline}
F^2 + F = 
  \frac{b^2}{a^2} x^4 
  + \left( \frac{B_1 b^2}{a^2} + \frac{a^2c^2}{d^4}\right) x^2 \\
  + \left( \frac{b^8}{a^8} + \frac{B_1 b^4}{a^4} + \frac{b^2c^2}{d^4}\right) x 
  + \left( \frac{b^{10}}{a^{10}} + \frac{B_1 b^6}{a^6} + \frac{e^2c^2}{d^4}\right).
\end{multline}
Therefore $F = f_2 x^2 + f_1 x + f_0$ for some $f_2, f_1, f_0\in k$. Looking at
the coefficient of $x^4$ in the preceding equality, we find that $f_2 = b/a$. 
Looking at the coefficient of~$x$, we find
\[
f_1 =   \frac{b^8}{a^8} + \frac{B_1 b^4}{a^4} + \frac{b^2c^2}{d^4}.
\]
We also have the condition that $\beta$ is an involution, which implies that 
\[
f_2 x^2 + f_1 x + f_0 
  + f_2 \left(x+\frac{b^2}{a^2}\right)^2 
  + f_1 \left(x+\frac{b^2}{a^2}\right) 
  + f_0 = 0;
\]
this simplifies to $a^2 f_1 + b^2 f_2 = 0$. Combining these three relations, we
find that
\begin{equation}
\label{EQ:oldrel1a}
0 = B_1 a^4 b^2 d^4 + a^8 c^2 + a^5 b d^4 + b^6 d^4.
\end{equation}

The same argument, applied to the involution~$\alpha$, shows that we must have
\begin{equation}
\label{EQ:oldrel2a}
0 = B_2 b^4 c^4 d^2 + a^2 c^8 + b^4 c^5 d + b^4 d^6.
\end{equation}

Next we use the fact that the extension $D\to\PP^1$ of the top $\PP^1$ in
Diagram~\eqref{EQ:V4bottomwithD} is obtained both from $C_1 \to \PP^1$ and from 
$C_2\to\PP^1$. This simply means that the Artin--Schreier extension of $k(t)$ we
get from~\eqref{EQ:x2} and $y^2 + y = x^5 + B_1 x^3$ is isomorphic to the
Artin--Schreier extension we get from~\eqref{EQ:u2} and 
$v^2 + v = u^5 + B_2 u^3$. This implies that the expression 
\[ 
R = x^5 + B_1 x^3 +  u^5 + B_2 u^3, 
\]
viewed as a polynomial in~$t$, can be written $S^2 + S$ for some $S\in k[t]$. 
If we write
\[ 
R = s_{10} t^{10} + s_9 t^9 + s_8 t^8 + \cdots + s_1 t + s_0, 
\]
then the equality $R = S^2 + S$ implies (among other things) that
\begin{align}
\label{EQ:rel9} 0 &= s_9\\
\notag          0 &= s_8 + s_4^2 + s_2^4 + s_1^8.
\end{align}

With a computer algebra system it is easy to check\footnote{
     See Verification 5.8c in the Magma file.
} that 
\[ 
s_9 = \frac{(a^2 b + c^2 d)^2}{(a d^2 + b^2 c)^5}, 
\]
so~\eqref{EQ:rel9} is equivalent to  $a^2 b = c^2 d$. Since we have scaled our
coefficients so that $d = a^2$, it follows that we also have $b = c^2$. Now we 
use these equalities to eliminate $b$ and $d$ from our equations. First we note
that the nonsingularity of the conic defined by~\eqref{EQ:conic2} implies that
$a^5 \ne c^5$. Second, using the fact that $a$ and $c$ are nonzero, we find
that~\eqref{EQ:oldrel1a} becomes~\eqref{EQ:B1} and~\eqref{EQ:oldrel2a} 
becomes~\eqref{EQ:B2}.

Using our new scaled variables we find\footnote{
     See Verification 5.8d in the Magma file.
} that
\[ 
s_8 + s_4^2 + s_2^4 + s_1^8 
  = \frac{(p_{16}e^{16} + p_8 e^8 + p_4 e^4 + p_2 e^2 + p_1 e + p_0)^2}
         {a^{16} c^{16} (a^5 + c^5)^8},
\]
where the coefficients $p_i$ are as in~\eqref{EQ:pvalues}. Therefore $e$ is a 
root of~\eqref{EQ:additive}.

Let $r = a^{40}$ and $s = c^{40}$. Then $a$, $c$, $e$, $B_1$, and $B_2$ are
exactly as in the construction presented at the beginning of the section, 
applied to $r$ and~$s$. 

We see that every diagram in the form of Diagram~\eqref{EQ:dihedral} that
contains our given $C_1$ and $C_2$ comes from an $(r,s)$ pair with $r$ and $s$
nonzero and distinct. Furthermore, since $r$ and $s$ can be recovered from the
diagram, distinct pairs give rise to distinct diagrams. The theorem follows.
\end{proof}

The preceding theorem is phrased in terms of the invariants $I_1$ and $I_2$ of
the curves $C_1$ and $C_2$. There is an equivalent statement that involves
specifying models for $C_1$ and~$C_2$.

\begin{corollary}
\label{C:DiagramGeneral}
Let $C_1$ and $C_2$ be supersingular genus-$2$ curves given by equations
$y^2 + y = x^5 + B_1 x^3$ and $y^2 + y = x^5 + B_2 x^3$. Up to isomorphism, 
diagrams in the form of Diagram~\textup{\eqref{EQ:dihedral}} containing this 
$C_1$ and $C_2$ are in bijection with pairs $(a,c)$ of nonzero elements of $k$
such that $a^5\ne c^5$ and
\begin{equation}
\label{EQ:Brelations}
\left\{\ 
\begin{aligned}
a^5 + c^{10} + 1 &= a^4 c^2 B_1\\
a^{10} + c^5 + 1 &= a^2 c^4 B_2.
\end{aligned}
\right.
\end{equation}
\end{corollary}

\begin{proof}
The first step of the proof of Theorem~\ref{T:DiagramGeneral} was to choose 
$B_1$ and $B_2$ in $k$ with $B_1^{40} = I_1$ and $B_2^{40} = I_2$, so that
$y^2 + y = x^5 + B_i x^3$ is a model for the curve with invariant~$I_i$. Then
the proof showed how dihedral diagrams containing the given curves are in 
bijection with pairs $(a,c)$ with $a^5\ne c^5$ satisfying~\eqref{EQ:Brelations}.
\end{proof}

In the following section we will use Theorem~\ref{T:DiagramGeneral} and its
corollary to count Richelot isogenies between supersingular curves. At this
point, however, we are already able to see an easy consequence of the theorem.

\begin{corollary}
\label{T:DiagramTwoSpecial}
There are no diagrams in the form of Diagram~\textup{\eqref{EQ:dihedral}} when
$C_1$ and $C_2$ both have supersingular invariant~$0$.
\end{corollary}

\begin{proof}
Theorem~\ref{T:DiagramGeneral} says that such diagrams correspond to pairs
$(r,s)$ of distinct nonzero elements of $k$ with $r = s^2 + 1$ and 
$s = r^2 + 1$. Suppose we have such a pair. Combining the two equalities we find
that $r^4 + r = 0$, and since $r\ne 0$ we must have $r^3 = 1$. We cannot have
$r=1$ because then $s = 0$, so $r$ must be a primitive cube root of unity. But
then from $s = r^2 + 1$ we see that $s = r$, which contradicts the requirement
that $r$ and $s$ be distinct.
\end{proof}

\subsection{Counting isogenies}
\label{SS:counting}
Now we are in a position to show that the dihedral construction, the degenerate 
construction, the Frobenius,  and the Verschiebung account for all of the
Richelot isogenies between two supersingular curves. First we look at the case 
of two curves whose supersingular invariants are nonzero, so that according to
Theorem~\ref{T:main} there are $60$ isogenies to account for.

\begin{theorem}
Let $C_1$ and $C_2$ be supersingular genus-$2$ curves over $k$ with 
supersingular invariants $I_1$ and $I_2$, respectively, and suppose $I_1$ and
$I_2$ are nonzero. Then the number of isomorphism classes of Richelot isogenies
from $C_1$ to $C_2$ coming from the dihedral construction, from the degenerate
construction, from Frobenius, and from Verschiebung are as given in 
Table~\textup{\ref{Table:Count}}, and these account for all of the Richelot 
isogenies from $C_1$ to~$C_2$.
\end{theorem}

\begin{table}[ht]
\caption{
The number of isomorphism classes of Richelot isogenies from a supersingular 
genus-$2$ curve with invariant $I_1\ne 0$ to a supersingular genus-$2$ curve 
with invariant $I_2\ne 0$ coming from the dihedral construction, from the 
degenerate construction, from Frobenius, and from Verschiebung, depending on
whether $I_1=I_2$, $I_1=I_2^2$, and $I_2 = I_1^2$. The fourth row gives the case
where $I_1$ and $I_2$ are distinct primitive cube roots of unity, and the sixth
row gives the case when $I_1 = I_2 = 1$.
}
\begin{center}
\begin{tabular}{ccc|cccc}
\toprule
\multicolumn{3}{c|}{Conditions on $I_1$ and $I_2$:} &
\multicolumn{4}{c}{Number of isogenies coming from$\ldots$}\\
$I_1=I_2$\textup{?} & $I_1 = I_2^2$\textup{?} & $I_2 = I_1^2$\textup{?} & 
Dihedral & Degen. & Frob. & Ver.\\
\midrule
 no &  no &  no & 60 & \pz0 &  0 &  0 \\
 no &  no & yes & 59 & \pz0 &  1 &  0 \\
 no & yes &  no & 59 & \pz0 &  0 &  1 \\
 no & yes & yes & 58 & \pz0 &  1 &  1 \\
yes &  no &  no & 50 &   10 &  0 &  0 \\
yes & yes & yes & 48 &   10 &  1 &  1 \\
\bottomrule
\end{tabular}
\end{center}
\label{Table:Count}
\end{table}

\begin{proof}
First let us consider the number of isogenies coming from the dihedral
construction. Pick values of $B_1$ and $B_2$ in $k$ so that $B_1^{40} = I_1$ and
$B_2^{40} = I_2$. From Proposition~\ref{P:dihedral} and 
Corollary~\ref{C:DiagramGeneral} we know that the number of isomorphism classes
of Richelot isogenies from $C_1$ to $C_2$ is given by the number of pairs 
$(a,c)$ of nonzero elements of $k$ with $a^5 \ne c^5$ that 
satisfy~\eqref{EQ:Brelations}.

Note that  the conditions that $I_1=I_2$ or $I_1 = I_2^2$ or $I_2 = I_1^2$ 
become the conditions that $B_1^5 = B_2^5$ or $B_1^5 = B_2^{10}$ or 
$B_2^5 = B_1^{10}$, respectively.

Let us view $a$, $c$, $B_1$, and $B_2$ as indeterminates, and let $f_1$, $f_2$,
and $g$ be the polynomials in $\BF_2[a,c,B_1,B_2]$ given by
\begin{align*}
f_1 &= B_1 a^4 c^2 + a^5 + c^{10} + 1\\
f_2 &= B_2 a^2 c^4 + a^{10} + c^5 + 1\\
  g &= B_1^2 a^8 + B_2 a^2 + c^{16} + c.
\end{align*}
Note that $c^4 g = f_1^2 + f_2$, so every set of nonzero values of $a$, $c$,
$B_1$, and $B_2$ that satisfies $f_1$ and $f_2$ also satisfies~$g$, and every
set of nonzero values that satisfies $f_1$ and $g$ also satisfies~$f_2$.

Given two elements $h_1$, $h_2$ in $\BF_2[a,c,B_1,B_2]$ and a variable 
$v\in \{a,c,B_1,B_2\}$, we let $\Res_v(h_1,h_2)$ denote the resultant of $h_1$ 
and $h_2$ with respect to the variable~$v$ and we let $\Disc_v(h_1)$ denote the 
discriminant of $h_1$ with respect to the variable~$v$.  We compute\footnote{
     See Verification 5.11a in the Magma file.
} that
\begin{equation}
\label{EQ:a}
\Res_c(f_1,g_1)  = a^{20} R,
\end{equation}
where
\begin{equation}
\label{EQ:r}
\begin{aligned}
 R &= a^{60} + a^{45} + a^{40} + B_1 B_2^2 a^{33} + B_1 B_2^2 a^{28} + a^{25} + B_1^2 B_2^4 a^{21}\\
   &\qquad    + a^{20} + B_1 B_2^2 a^{18} + B_1^2 B_2^4 a^{16} + B_1^4 B_2^8 a^{12} + B_1^3 B_2^6 a^9 + B_1 B_2^2 a^8 \\
   &\qquad    + B_1^2 B_2^4 a^6 + a^5 + B_1^3 B_2^6 a^4 + B_1 B_2^2 a^3 + B_1^2 B_2^4 a + (B_1^5 + B_2^{10}).
\end{aligned}
\end{equation}
We also check that 
\begin{equation}
\label{EQ:disca}
\Disc_a R = (B_1^5 + B_2^{10})^{44}.
\end{equation}

Suppose now we are given specific nonzero values of $B_1$ and $B_2$ in~$k$ with
$B_1^5 \ne B_2^{10}$ and $B_2^5\neq B_1^{10}$ and $B_1^5\neq B_2^5$. 
From~\eqref{EQ:a} and~\eqref{EQ:disca} we find that there are exactly $60$
nonzero values of $a\in k$ for which there exists a $c\in k$ with 
$f_1(a,c) = g(a,c) = 0$. We claim that each of these values of $c$ is nonzero.
To see this, note that if $c$ were $0$ we would have 
$0 = f_1(a,0) = a^5 + 1$ and $0 = g(a,0) = B_1^2 a^8 + B_2 a^2$; combining these
equalities, we find that $B_1^{10} = B_2^5$, a contradiction. Thus there are 
exactly $60$ pairs $(a,c)$ of nonzero elements of $k$ that satisfy
$f_1(a,c) = f_2(a,c) = 0$.  We must still consider whether any of these pairs
fails to satisfy the requirement that $a^5\ne c^5$. But if $a^5 = c^5$ and 
$f_1(a,c) = f_2(a,c) = 0$, then we find that $B_1^5 = B_2^5$, a contradiction.
This gives us the count of dihedral Richelot isogenies on the first line of 
Table~\ref{Table:Count}.

Now consider the case when $I_1 \neq I_2^2$ and $I_2 \neq I_1^2$ but 
$I_1 = I_2$. In this case we choose our values $B_1$ and $B_2$ to be equal, and
we repeat the computations we above, but this time we work in the polynomial
ring $\BF_2[a,c,B]$ and set $B_1 = B_2 = B$. (Note that the assumptions that
$I_1 \neq I_2^2$ and $I_2 \neq I_1^2$  imply that $I_1\ne 1$, so $B^5\neq 1$.)
The same argument as before shows that there are $60$ pairs $(a,c)$ of nonzero
elements of $k$ that satisfy $f_1$ and $f_2$, but now for some number of these
pairs we have $a^5 = c^5$. We count these unacceptable pairs as follows.

Suppose we have a solution pair $(a,c)$ with $a^5 = c^5$. Write $c = a\zeta$,
where $\zeta^5=1$. Since $f_1(a,c) = f_2(a,c) = 0$ and $B_1 = B_2$, we find that
$\zeta^2 = \zeta^4$, so $\zeta = 1$. Thus, the bad $(a,c)$ pairs correspond to 
the values of $a$ such that $f_1(a,a) = 0$; that is, the roots of 
$T^{10} + B T^6 + T^5 + 1$. This polynomial has discriminant~$1$, so it has $10$
distinct roots, all nonzero.\footnote{
     See Verification 5.11b in the Magma file.
} This shows that of the $60$ pairs $(a,c)$ of nonzero elements satisfying $f_1$
and~$f_2$, exactly $10$ fail to meet the condition that $a^5 \neq c^5$. This
gives us the count of dihedral Richelot isogenies on the fifth line of 
Table~\ref{Table:Count}.

Now we consider the case where $I_2 = I_1^2$, but $I_1\neq I_2$ and 
$I_1\neq I_2^2$. We choose our values $B_1$ and $B_2$ so that $B_2 = B_1^2$.
Again we find $60$ distinct nonzero values of $a$ for which there exists a
$c\in k$ with $f_1(a,c) = g(a,c) = 0$. However, we check that now there
\emph{is} a (unique) nonzero value of $a$ satisfying $f_1(a,0) = g(a,0) = 0$,
namely $a = 1$. This leaves us with $59$ pairs $(a,c)$ of nonzero elements 
satisfying $f_1(a,c) = f_2(a,c) = 0$, and once again the assumption that
$B_1^5 \ne B_2^5$ shows that $a^5\neq c^5$ for each of these pairs. This gives
us the count of dihedral Richelot isogenies on the second line of 
Table~\ref{Table:Count}, and the entry on the third line follows by symmetry.

We turn to the case where $I_1 = I_2^2$ and $I_2 = I_1^2$, but where 
$I_1 \neq I_2$. This means that $I_1$ and $I_2$ are distinct primitive cube
roots of unity, and we can choose $B_1$ and $B_2$ so that they too are distinct
primitive cube roots of unity. We can let $a$ run through the nonzero roots of
the polynomial $R$ from~\eqref{EQ:r} and let $c$ run through the nonzero roots
of the corresponding polynomial with the roles of $B_1$ and $B_2$ reversed, and
check to see which pairs $(a,c)$ satisfy $f_1(a,c) = f_2(a,c) = 0$ and
$a^5 \neq c^5$. We find exactly $58$ pairs\footnote{
     See Verification 5.11c in the Magma file.
}, and this gives us the count of dihedral Richelot isogenies on the fourth 
line of Table~\ref{Table:Count}.

Finally, we look at the case where $I_1 = I_2 = 1$. As in the preceding case, we
can explicitly calculate the pairs $(a,c)$ of nonzero elements of $k$ such that
$f_1(a,c) = f_2(a,c) = 0$ and $a^5 \neq c^5$, and we find there are $48$ such
pairs.\footnote{
     See Verification 5.11d in the Magma file.
} This gives us the count of dihedral Richelot isogenies on the last line of 
Table~\ref{Table:Count}.

We turn now to the question of counting degenerate Richelot isogenies. 
Degenerate isogenies require that $C_1\cong C_2$; that is, that $I_1 = I_2$.
This explains the four zero entries in the ``Degenerate'' column of 
Table~\ref{Table:Count}.

Suppose $I_1 = I_2$, and let $C$ be the curve $C_1\cong C_2$. We see from 
Proposition~\ref{P:degenerate} that the number of degenerate Richelot isogenies
is equal to the number of conjugacy classes of elements $\gamma\in \Aut C$ with
$\gamma^2 = \iota$, the hyperelliptic involution. Statement~2(c) of 
Lemma~\ref{L:models} says that there are exactly $10$ such conjugacy classes, 
and this gives us the count of degenerate Richelot isogenies on the last two 
lines of the table.

The entries for Frobenius and Verschiebung are self-explanatory.

We see that for each pair $(I_1,I_2)$ we can account for $60$ distinct Richelot 
isogenies, the number given by Theorem~\ref{T:main}.
\end{proof}

\begin{theorem}
Let $C_1$ and~$C_2$ be supersingular genus-$2$ curves over $k$ with 
supersingular invariants $I_1\neq 0$ and $I_2 = 0$, respectively. Then up to 
isomorphism, there are $12$ Richelot isogenies from $C_1$ to $C_2$ coming from 
the dihedral construction, and these account for all of the Richelot isogenies 
from $C_1$ to~$C_2$.
\end{theorem}

\begin{proof}
From Proposition~\ref{P:dihedral} and Theorem~\ref{T:DiagramGeneral} we know
that the number of isomorphism classes of Richelot isogenies from $C_1$ to $C_2$
coming from the dihedral construction is equal to the number of pairs $(r,s)$ of
distinct nonzero element of $k$ such that $r^2 + s + 1 = 0$ and 
$(r + s^2 + 1)^5 = r^4 s^2 I_1$. Plugging $s = r^2 + 1$ into the second
equality, we find that $r^4 (r+1)^4 f = 0$, where 
\[
f \colonequals r^{12} + r^9 + r^8 + r^5 + r^4 + r + I_1 = r(r+1)(r^2 + r + 1)^5 + I_1.
\]
Since we only want to count $(r,s)$ pairs where $r$ and $s$ are nonzero, we can 
ignore the solutions where $r = 0$ or $r = 1$, and look only at the roots 
of~$f$. The discriminant of $f$ is $I_1^8\neq 0$, and the constant term of $f$
is also nonzero, so we find\footnote{
     See Verification 5.12 in the Magma file.
} there are $12$ possible nonzero values of $r$ --- and none of them is equal 
to~$1$, so the corresponding values of $s$ are nonzero. The last thing we must
check is that for the roots $r$ of $f$, we do not have $r = r^2 + 1$, and this 
is clear from the rightmost expression for~$f$.

Thus we have $12$ pairs $(r,s)$ that give rise to distinct dihedral Richelot
isogenies from  $C_1$ to  $C_2$, and by Theorem~\ref{T:count} this accounts for 
all of the Richelot isogenies from $C_1$ to~$C_2$.
\end{proof}

\begin{theorem}
Let $C$ be the supersingular genus-$2$ curve over $k$ with supersingular 
invariant~$0$. There are exactly $2$ isomorphism classes of Richelot isogeny
from $C$ to itself coming from the degenerate construction, and these two 
degenerate isogenies, together with the Frobenius and the Verschiebung,
represent all of the isomorphism classes of Richelot isogenies from  $C$ to
itself.
\end{theorem}

\begin{proof}
Corollary~\ref{T:DiagramTwoSpecial} says that there are no Richelot isogenies
from $C$ to itself coming from the dihedral construction. 
Proposition~\ref{P:degenerate} and part 3(c) of Lemma~\ref{L:models} show that 
there are exactly $2$ isomorphism classes of Richelot isogeny from $C$ to
itself coming from the degenerate construction. The Frobenius and the
Verschiebung are also Richelot isogenies from $C$ to itself, and 
Theorem~\ref{T:FneqV} says that they are not isomorphic to one another. This
gives us a total of four isomorphism classes of Richelot isogeny from  $C$ to
itself, and Theorem~\ref{T:main} says that there are no others.
\end{proof}

\bibliography{Richelot}
\bibliographystyle{hplaindoi}
\end{document}